\documentclass[11pt]{amsart} \usepackage{fullpage}
\usepackage{amsfonts,amssymb,mathrsfs,amsmath,amscd,epsfig,setspace,amstext,color,graphics,verbatim,mathtools}
\usepackage{paralist,hyperref} \usepackage{pgf} \usepackage{tikz}
\usepackage{frenchineq} \usetikzlibrary{calc} \usetikzlibrary{arrows}
\usetikzlibrary{shapes} \usepackage{barycentric}

\providecommand{\doi}[1]{Eprint \href{http://dx.doi.org/#1}{doi:#1}}
\providecommand{\arxiv}[2][]{\href{http://www.arXiv.org/abs/#2}{arXiv:#2}}
\DeclareSymbolFontAlphabet{\mathbb}{AMSb}
\DeclareMathAlphabet{\mathbbb}{U}{bbold}{m}{n}

\def\interval[#1,#2]{[#1,#2]}
\newcommand{\NEW}[1]{{\em #1}}
\newcommand{\maxpolgraph}[2]{#2^{#1}}
\newcommand{\maxpolmap}[2]{#2^{#1}}
\newcommand{\cP}{\mathcal{P}}

\newcommand{\cR}{\mathcal{R}}
\newcommand{\cZ}{\mathcal{Z}}
\newcommand{\polar}[1]{{#1}^{\circ}}
\newcommand{\ipolar}[1]{{#1}^{\circ}_i}
\newcommand{\floor}[1]{\lfloor #1\rfloor}
\newcommand{\R}{\mathbb{R}}
\newcommand{\Q}{\mathbb{Q}}
\newcommand{\N}{\mathbb{N}}
\newcommand{\Z}{\mathbb{Z}}

\newcommand{\sK}{\mathcal{K}}

\newcommand{\sG}{\mathcal{G}}

\newcommand{\RowSpace}[1]{\cR(#1)}

\newcommand{\maxplus}{\R_{\max}}

\newcommand{\defi}{:=}%

\newcommand{\mpzero}{\mathbbb{0}}
\newcommand{\mpone}{\mathbbb{1}}

\newcommand{\setexcept}[2]{[#1\setminus #2]}
\newtheorem{theorem}{Theorem}
\newtheorem{proposition}[theorem]{Proposition}
\newtheorem{lemma}[theorem]{Lemma}

\newtheorem{corollary}[theorem]{Corollary}
\newtheorem{assumption}{Assumption}

\theoremstyle{definition}
\newtheorem{definition}{Definition}
\theoremstyle{remark}
\newtheorem{remark}{Remark}
\newtheorem{example}{Example}

\newcommand{\argmax}{\arg\max}
\newcommand{\uvector}{\makebox{\rm e}}
\newcommand{\mon}{-}
\newcommand{\mop}{+}

\title{Tropical polar cones, hypergraph transversals, and mean payoff games}
\author{Xavier {A}llamigeon}
\address[X.~Allamigeon and S.~Gaubert]{INRIA and CMAP, \'Ecole Polytechnique, 91128 Palaiseau Cedex France}
\email[X.~Allamigeon]{xavier.allamigeon@inria.fr}
\author{{S}t{\'e}phane {G}aubert}
\email[S.~Gaubert]{stephane.gaubert@inria.fr}
\author{Ricardo D. Katz}
\address[R.D.~Katz]{CONICET. Postal address:\!\! Instituto de Matem\'atica 
``Beppo Levi'',\! Universidad Nacional de Rosario, 
Avenida Pellegrini 250, 2000 Rosario, Argentina.}
\email{rkatz@fceia.unr.edu.ar}
\date{April 15, 2010, revised October 27, 2010}
\catcode`\@=11
\@namedef{subjclassname@2010}{%
  \textup{2010} Mathematics Subject Classification}
\catcode`\@=12
\thanks{This work was performed when the first author was with EADS Innovation Works, SE/IA -- Suresnes, France and CEA, LIST MeASI -- Gif-sur-Yvette, France.\\
It was partially supported by the Arpege programme of the French National Agency of Research (ANR), project ``ASOPT'', number ANR-08-SEGI-005 , by the Digiteo project
DIM08 ``PASO'' number 3389, and by the French-Romanian LEA ``Mathmode'' cooperation programme.}
\begin{document}
\begin{abstract}
We discuss the tropical analogues of several basic questions of convex
duality. In particular, the polar of a tropical polyhedral cone represents the set of linear
inequalities that its elements satisfy. 
 We characterize the extreme rays
of the polar in terms of certain minimal set covers
which may be thought of as weighted generalizations of minimal transversals
in hypergraphs. We also give a tropical analogue of Farkas lemma,
which allows one to check whether a linear inequality is implied
by a finite family of linear inequalities. Here, the certificate
is a strategy of a mean payoff game. We discuss examples, 
showing that the number of extreme rays of the polar of the tropical cyclic polyhedral cone is polynomially bounded, and that there is no unique minimal system of inequalities
defining a given tropical polyhedral cone.
\end{abstract}
\maketitle

\section{Introduction}
The max-plus or tropical analogue of classical convexity has emerged
in a number of works. Early contributions come back to Zimmermann~\cite{zimmerman77} and Cuninghame-Green~\cite{cuni79}.
The analogues of cones (thought of as modules over the tropical semiring)
have been studied by Litvinov, Maslov, and Shpiz~\cite{litvinov00}
as part of ``Idempotent analysis'', and by Cohen, Gaubert and Quadrat~\cite{cgq02}, motivated by discrete event systems. Relations with abstract convexity
have appeared in a further work with Singer~\cite{cgqs04}, and in the work of Briec and Horvath~\cite{BriecHorvath04}. The interest in the subject has been renewed after the work of Develin and Sturmfels~\cite{DS}, who developed a combinatorial approach motivated by tropical geometry. This was at the origin
of a number of works of the same authors and of Joswig, Yu, and Block, see~\cite{joswig04,JSY07,blockyu06}. Some recent developments include~\cite{BSS,katz08,joswig-2008,GM08,AGG10}.

In classical convex analysis, duality techniques play an important role, and,
in the light of the current development of tropical convexity,
it is natural to ask whether these carry over to the tropical setting.

In particular, a polyhedral cone can be classically represented in two different ways, either 
internally, ``by generators,'' as the set of nonnegative linear combinations of a finite set of vectors, or externally, ``by relations,'' as the solution set of a finite system
of homogeneous linear inequalities. 
If the cone is pointed, among the generating sets of vectors, there turns
out to be a unique minimal one (up to a scaling of each vector),
which consists of representatives of the extreme rays of the cone.
If the cone is of full dimension, among the defining systems
of inequalities, there turns out to be a unique minimal one (up to a scaling
of each inequality), corresponding to the facets of the cone, or, 
by duality, to the extreme rays of its polar.
Moreover, by Farkas lemma, every (homogeneous, linear) inequality
verified by all the vectors of the cone can be obtained by taking a nonnegative linear combination of the ``facet defining'' inequalities.

We may ask for the tropical analogues of these properties.
For the internal representation, precisely the same
result holds: the tropical analogue
of the Minkowski theorem~\cite{GK06a,BSS,GK} shows that a tropical
polyhedral cone is generated by its extreme rays, and that
every generating set must contain one representative
of each extreme ray.

The external representation has also been studied in the
tropical setting. The analogue of the polar, which consists of the set
of inequalities verified by the elements of the cone,
was introduced in~\cite{katz08}. 
The polar of a tropical polyhedral cone is in fact a tropical
polyhedral cone in a space of double dimension. 
It has a finite family of extreme rays, which determine a finite family
of ``extreme'' inequalities, having the property that every
inequality verified by all the vectors of the cone is a linear combination
of the extreme inequalities. Moreover, the set of extreme inequalities
is the unique minimal one having this property (up to a scaling of
each inequality). 

In this paper, we pursue the investigation of tropical analogues of classical properties concerning convex duality, and in particular polars.

First, we establish (Theorem~\ref{TheorExtremeMinimal} and 
Proposition~\ref{prop-corres} below) 
a characterization of the extreme rays of the polar in terms of 
minimal set covers. The latter may be thought of as weighted generalizations
of minimal transversals (or hitting sets) in hypergraphs.
It follows that enumerating the extreme rays of the polar is at least as
hard as the well known problem of enumerating the minimal transversals of an hypergraph (Corollary~\ref{CorTransversal} and Proposition~\ref{CorEquivExtPolars}).
Moreover, by combining this characterization
with a result of Elbassioni~\cite{Elbassioni08}, building
on a line of works on minimal transversals and minimal solutions
of monotone inequalities by Fredman, Khachiyan, Boros, Elbassioni,
Gurvich, and Makino~\cite{Fredman,BEGKM02,KBEG06},  
it follows that the set of extreme rays of the
polar of a tropical polyhedral cone can be computed
in incremental quasi-polynomial time.

Next, we answer to a question raised in~\cite{katz08}, which asks
for a tropical analogue of Farkas lemma. Indeed, it was observed there
that the statement of the classical Farkas lemma is no
longer valid in the tropical world:
there are inequalities which can be logically deduced from 
a finite system of inequalities but which cannot be obtained
by taking linear combinations of the inequalities in this system
(see Figure~\ref{FigDeduc} below for an example).

We show here that there is indeed a tropical analogue of Farkas
lemma, if we consider properly its role. The classical result
provides a certificate (nonnegative weights, or Lagrange multipliers),
which allows one to easily check (by computing a linear combination)
that an inequality follows from other inequalities. The tropical analogue,
Theorem~\ref{theo-farkas} below, 
shows that there is still a concise certificate,
which is no longer a collection of Lagrange multipliers, but 
consists of a strategy of a mean payoff game. 
It also follows from our approach that checking whether an inequality is
implied by a finite family of inequalities is polynomial
time equivalent to the problem of solving a mean payoff
game
(the latter was already known to be in {\sc NP} $\cap$ {\sc co-NP}, 
and the existence
of a polynomial time algorithm for this problem is a long standing open question).

Theorem~\ref{theo-farkas} relies on a recent work of Akian, 
Gaubert and Guterman~\cite{AGG,aggut10}, 
setting up a correspondence between mean payoff games 
and external representations of tropical polyhedra, 
and showing that the value of a mean payoff game is nonnegative if, 
and only if, the corresponding tropical polyhedron is non-empty. 
In the present paper, we use similar techniques to show that the 
decision problem associated with the tropical Farkas lemma is also 
equivalent to a mean payoff game problem (Corollary~\ref{coro-equiv}).

We also discuss examples, computing in particular the 
extreme rays of the polar of the tropical analogues of cyclic polyhedral cones. 
Whereas in classical algebra, cyclic polyhedral cones
have an exponential number of facets (actually, they maximize the numbers of
facets among all the cones of the same
dimension and with the same number of extreme rays), in the tropical
setting their polars turn out to have only a polynomial number
of extreme rays, see Proposition~\ref{PropExtCyclic} below. 
We finally show that unlike in the classical case,
there is no unique minimal set of inequalities defining a given
tropical polyhedral cone.

\section{Extreme elements of the polar}\label{SecExtremePolar}
 
We start by recalling some basic definitions and notation. 
Tropical or max-plus algebra is the analogue of classical linear algebra 
developed over the max-plus semiring $\maxplus$, 
which is the set $\R\cup\{-\infty\}$ equipped with the 
addition $(a,b)\mapsto \max(a,b)$ and the multiplication $(a,b)\mapsto
a+b$. To emphasize the semiring structure, we write $a\oplus b:=\max(a,b)$,
$ab:=a+b$, $\mpzero:=-\infty$ and $\mpone:=0$.  
The semiring operations are extended in 
the usual way to matrices over the max-plus semiring: 
$(A\oplus B)_{ij}:=A_{ij}\oplus B_{ij}$, 
$(AB)_{ij}:=\oplus_k A_{ik} B_{kj}$ and $(\lambda A)_{ij}:=\lambda A_{ij}$ 
for all $i,j$, 
where $A,B$ are matrices of compatible sizes and $\lambda \in \maxplus$. 
In what follows, we shall denote by $G_{\cdot i}$ 
(resp. $G_{i \cdot }$) the $i$th column (resp. row) of the matrix $G$,  
and by $\uvector^i$ the $i$th vector of the canonical basis of $\maxplus^n$, 
i.e. the vector defined by $(\uvector^i)_j\defi \mpone $ if $j=i$ and 
$(\uvector^i)_j\defi \mpzero $ otherwise. 

Several classical concepts and results have their tropical analogues. 
In particular, the tropical analogues of convex sets were introduced 
by Zimmermann~\cite{zimmerman77}. Since in the max-plus semiring 
any scalar is ``positive'', i.e. for any $\lambda \in \maxplus$ 
we have $\lambda \geq \mpzero $, it is natural to define the  
{\em tropical segment} joining two points $x,y\in \maxplus^n$ 
as the set of points of the form $\lambda x \oplus \mu y$ where 
$\lambda$ and $\mu $ are elements of $\maxplus $ such that 
$\lambda \oplus \mu = \mpone $. Then, a subset of $\maxplus^n$ 
is said to be a \NEW{tropical convex set} if it contains any tropical 
segment joining two of its points. Similarly, the
tropical cone generated by $x,y\in \maxplus^n$ is defined as the 
set of vectors of the form $\lambda x \oplus \mu y$ 
where now $\lambda$ and $\mu $ are arbitrary elements of $\maxplus $.  
A \NEW{tropical (convex) cone} is a subset $\sK$ of $\maxplus^n$ 
which contains any tropical cone generated by two of its vectors. 
An example of tropical convex set is given on the 
left hand side of Figure~\ref{TropicalConvexity}:
the tropical convex set is the closed gray region together with the
horizontal segment joining the point $u$ to it. Three tropical
segments in general position are represented in bold. 
Comparing the shapes of these segments with the shape of the set, 
one can check geometrically that this set is a tropical convex set. 
A tropical cone generated by two vectors $u,v\in \maxplus^2$ 
is represented on the right hand side of Figure~\ref{TropicalConvexity} 
by the unbounded gray region.  

\begin{figure}
\begin{center}
\input{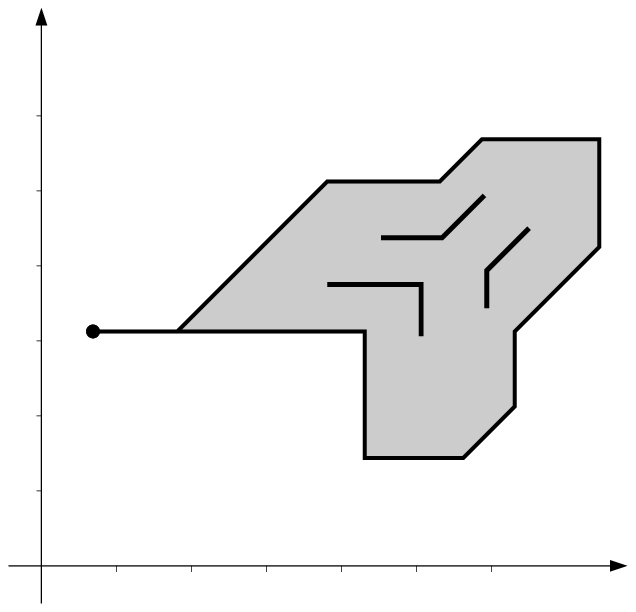}
\quad
\input{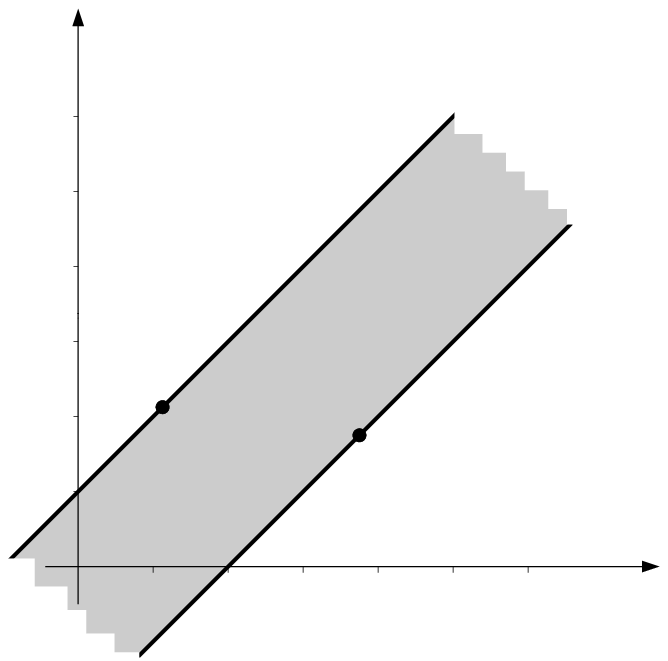}
\end{center}
\caption{Illustration of tropical convexity in $\maxplus^2$. 
(Left): A tropical convex set and three tropical 
segments in general position. (Right): The tropical cone generated 
by the vectors $u$ and $v$.}
\label{TropicalConvexity}
\end{figure}
  
Given a tropical cone $\sK \subseteq \maxplus^n$, 
we shall say that a non-identically $\mpzero$ vector $x\in \sK$ is 
\NEW{extreme} in $\sK$ if $x=y\oplus z$ with $y,z\in \sK$ 
implies $x=y$ or $x=z$. The set of scalar 
multiples of an extreme vector of $\sK$ is an \NEW{extreme ray} of $\sK$. 
A tropical cone $\sK$ is said to be  \NEW{finitely generated} (or \NEW{polyhedral}) 
if it contains a finite set of vectors $\left\{ y^r \right\}_{r\in [p]}$ 
such that every vector $x \in \sK$ can be expressed as
a \NEW{tropical linear combination} of these vectors, 
meaning that $x=\oplus_{r\in[p]} \lambda_r y^r$ 
for some $\lambda_r \in \maxplus $. Here, and in the sequel, 
we set $[p]:=\{1,\ldots,p\}$ for any $p\in \N$.
 
The \NEW{polar} of a subset $\sK\subseteq \maxplus^n$ can be defined~\cite{katz08} as 
\begin{equation}
\polar{\sK}\defi \left\{ (a,b)\in (\maxplus^n)^2 \mid a x \leq b x ,\; \forall x\in \sK\right\} \; , 
\end{equation}
where $x y \defi \oplus_{i\in[n]} x_i y_i$ for any pair of vectors $x,y\in \maxplus^n$. 
Note that the polar of $\sK$ is a tropical cone of $(\maxplus^n)^2$ which represents the set of 
(tropical) linear inequalities satisfied by the elements of $\sK$. 
In what follows, we shall usually identify an element $(a,b)$ of $\polar{\sK}$ 
with the corresponding inequality $a x \leq b x$, 
which will be called \NEW{valid} for $\sK$ because it is satisfied by all its elements. 

In this paper, we are mainly interested in polyhedral cones. 
Therefore, we shall consider the case in which the tropical cone $\sK$ 
is the \NEW{row space} $\RowSpace{G}$ of some $p\times n$ matrix $G$ 
with entries in $\maxplus$, i.e.
\begin{equation}
\RowSpace{G}:=\left\{ x\in \maxplus^n \mid 
x= \oplus_{r\in[p]} \lambda_r G_{r \cdot }\; , \;
 \lambda_r \in \maxplus \makebox{ for } r\in[p] \right\} \; , 
\end{equation}
where without loss of generality  
we assume that $G$ does not have an identically $\mpzero $ column. 
Observe that in this case we have
\[
\polar{\RowSpace{G}}=\left\{ (a,b)\in (\maxplus^n)^2 \mid G a \leq G b \right\} \; , 
\]
which implies that $\polar{\RowSpace{G}}$ has a finite number of extreme rays, 
because it can be expressed as the solution set of a two sided (tropical) 
linear system of equations, namely 
\[
\polar{\RowSpace{G}}=\left\{ (a,b)\in (\maxplus^n)^2 \mid G a \oplus G b = G b \right\} \; ,  
\] 
and the solution sets of such systems are known to be finitely generated 
tropical cones (see~\cite{butkovicH,gaubert92a,maxplus97}). 
We refer the reader to~\cite{GK09,katz08,GK06a} for more information. 

By the separation theorem for closed cones 
of~\cite{zimmerman77,shpiz,cgqs04}, 
it follows that $\RowSpace{G}$ is characterized by its polar cone, i.e.
\[
\RowSpace{G} =\left\{x\in \maxplus^n \mid a x \leq b x\; , \; \forall (a,b)\in \polar{\RowSpace{G}} \right\}  \; .
\]
This implies that $\RowSpace{G}$ can be expressed as 
the solution set of the (finite) set of linear inequalities 
associated with the extreme rays of $\polar{\RowSpace{G}}$. 
More precisely, the extreme rays of $\polar{\RowSpace{G}}$ 
determine a finite family of linear inequalities defining $\RowSpace{G}$,  
which has the property that any valid inequality for $\RowSpace{G}$ 
can be expressed as a (tropical) linear combination of the inequalities in this family. 
  
We shall say that a linear inequality in the variables $x_j$, $j\in [n]$,
is \NEW{trivial} if it is of the form $x_i\geq x_i$ or $x_i\geq \mpzero $, 
and that it is of \NEW{type $i$} if it is of the form 
\[
x_i \leq \bigoplus_{j\in\setexcept{n}{i}} z_j x_j \enspace ,
\]
where $z\in \maxplus^{n-1}$ and 
\[
\setexcept{n}{i}\defi [n] \setminus \{i\} = \{1,\ldots , n\}\setminus \{i\} \enspace .
\]

An inequality satisfied by the elements of a tropical cone is said to be 
\NEW{extreme} if it corresponds to an extreme vector of the polar of this cone.
A \NEW{system of representatives} of the extreme inequalities
is any family containing one and only one inequality proportional
to every extreme inequality.

\begin{proposition}\label{PropTypeI}
Every extreme inequality is either
proportional to a trivial inequality or proportional
to an inequality of type $i$, for some $i\in [n]$.
\end{proposition}

\begin{proof}
Firstly, note that any vector of the form $(\mpzero,\uvector^i)$, 
corresponding to the trivial inequality $x_i \geq \mpzero $, 
belongs to the polar of every cone. Therefore, 
these inequalities are clearly extreme and any extreme inequality 
corresponding to a vector of the form $(\mpzero, b)$, 
where $b\in \maxplus^n$, must be a scalar multiple of one of these inequalities. 
It can also be checked that the vectors $(\uvector^i,\uvector^i)$, 
corresponding to the trivial inequalities $x_i\leq x_i$, 
are also extreme because vectors of the form $(\uvector^i, \lambda \uvector^i)$, 
where $\lambda < \mpone$, do not belong to the polar of $\RowSpace{G}$ 
(recall that we assume that $G$ does not contain an identically $\mpzero $ column).  

Consider now the inequality $ax\leq bx$ corresponding to a 
vector $(a,b)$ of $\polar{\RowSpace{G}}$ with $a\neq \mpzero$. 
Since 
\[
(a'\oplus a'') x \leq b x \implies a' x \leq b x \mbox{ and } a'' x \leq b x \; ,
\] 
it follows that $a x\leq b x$ can only be extreme if $a$ is a scalar multiple of $\uvector^i$, for some $i\in [n]$. 
Therefore, without loss of generality, we may assume that $a=\uvector^i$. 
We next show that the inequality $\uvector^i x\leq b x$ is extreme only if 
$b=\uvector^i$ or $b_i=\mpzero $, 
which in the latter case means that $\uvector^i x\leq b x$ is of type $i$. 

Assume that $\uvector^i x\leq b x$ is not of type $i$, 
i.e.\ that $b_i\neq \mpzero$. If $b_i < \mpone $, 
then $(\uvector^i,\oplus_{j\in\setexcept{n}{i}} b_j \uvector^j)\in \polar{\RowSpace{G}}$ 
and as $(\uvector^i,b)= (\uvector^i,\oplus_{j\in\setexcept{n}{i}} b_j \uvector^j)\oplus (\mpzero, b_i \uvector^i)$, we conclude that 
$(\uvector^i,b)$ is not extreme. Analogously, if $b_i > \mpone $, since 
$(\uvector^i,b)= (\uvector^i ,\uvector^i )\oplus (\mpzero, b )$, it follows that neither in this case $(\uvector^i,b)$ is extreme. 
Finally, if $b_i = \mpone $, we have $(\uvector^i,b)= (\uvector^i,\uvector^i)\oplus (\mpzero, \oplus_{j\in\setexcept{n}{i}} b_j \uvector^j)$ 
implying that $(\uvector^i,b)$ is extreme only if $ \oplus_{j\in\setexcept{n}{i}} b_j \uvector^j=\mpzero$,  
i.e.\ only if $b=\uvector^i$. 
\end{proof}

\begin{definition}
The \NEW{$i$th polar} $\ipolar{\sK}$ of a subset $\sK \subseteq \maxplus^n$ is the tropical cone
\begin{equation} 
\ipolar{\sK} \defi \left\{ b\in \maxplus^n \mid b_i x_ i\leq \oplus_{j\in\setexcept{n}{i}} b_j x_j\; ,\; \forall x \in \sK \right\} \; .
\end{equation}
\end{definition}

Since $b\in \maxplus^n$ belongs to the $i$th 
polar $\ipolar{\sK}$ if, and only if,  
$(b_i\uvector^i,\oplus_{j\in\setexcept{n}{i}} b_j \uvector^j)\in (\maxplus^n)^2$ 
belongs to the polar $\polar{\sK}$, it follows that 
the extreme inequalities of type $i$ of 
$\sK$ correspond to the extreme rays of $\ipolar{\sK}$ associated 
with extreme vectors $b$ such that $b_i\neq \mpzero $. Moreover, 
as $\uvector^j\in \ipolar{\sK} $ for all $j\in\setexcept{n}{i}$, 
note that any extreme vector $b$ of $\ipolar{\sK} $ with 
$b_i = \mpzero $ must be a scalar multiple of one of these vectors 
of the canonical basis of $\maxplus^n$. Therefore, 
by Proposition~\ref{PropTypeI}, it follows that 
the study of the extreme inequalities of $\sK$ 
reduces to the study of the extreme rays of $\ipolar{\sK}$. 
In fact, the two underlying enumeration problems are polynomial time
equivalent, as shown by the following result.

\begin{proposition}\label{CorEquivExtPolars}
The enumeration problem for the extreme rays of the polar 
of a tropical polyhedral cone is polynomial time equivalent to 
the enumeration problem for the extreme rays of the $i$th polar 
of a tropical polyhedral cone. 
\end{proposition}

\begin{proof}
Let $\RowSpace{G}$ be a tropical polyhedral cone. As we already showed, 
the extreme inequalities of type $i$ of $\RowSpace{G}$ 
correspond to the extreme rays of $\ipolar{\RowSpace{G}}$ associated 
with extreme vectors $b$ such that $b_i\neq \mpzero $. Moreover, 
as $\uvector^j\in \ipolar{\RowSpace{G}} $ for all $j\in\setexcept{n}{i}$, 
any extreme vector $b$ of $\ipolar{\RowSpace{G}} $ with 
$b_i = \mpzero $ must be a scalar multiple of one of these vectors 
of the canonical basis of $\maxplus^n$. Therefore, 
by Proposition~\ref{PropTypeI}, it follows that 
the enumeration problem for the extreme rays of $\polar{\RowSpace{G}}$ 
reduces to the enumeration problem for the extreme rays of 
$\ipolar{\RowSpace{G}}$ for $i\in [n]$. 

Conversely, given a tropical polyhedral cone $\RowSpace{G}$, 
let $\sK\subseteq \maxplus^n$ be the tropical cone generated by 
(i.e. all the tropical linear combinations of) the rows of $G$ and 
the vectors $\uvector^j$ for all $j\in\setexcept{n}{i}$. 
Then, $b\in \maxplus^n$ belongs to the 
$i$th polar $\ipolar{\RowSpace{G}}$ of $\RowSpace{G}$ if, and only if, 
$(b_i\uvector^i,\oplus_{j\in\setexcept{n}{i}} b_j \uvector^j)\in (\maxplus^n)^2$ 
belongs to the polar $\polar{\sK}$ of $\sK$. Therefore, it follows that
$b$ is an extreme vector of $\ipolar{\RowSpace{G}}$ if, and only if,  
$(b_i\uvector^i,\oplus_{j\in\setexcept{n}{i}} b_j \uvector^j)$ 
is an extreme vector of $\polar{\sK}$.   
Moreover, any non-trivial extreme inequality of $\sK$ is of type $i$. 
In consequence, the enumeration problem for the extreme rays of 
the $i$th polar of $\RowSpace{G}$ reduces to the enumeration problem 
for the extreme rays of the polar of $\sK$. 
\end{proof}

\begin{example} 
Consider the row space $\RowSpace{G}$ of the matrix
\begin{equation}\label{MatrixG}
G=\left( 
\begin{array}{ccc}
-3 & 0  & 0  \\
0  & -3 & 0  \\
0  & 0  & -3 \\
1  & 0  & 0  \\
0  & 1  & 0  \\
0  & 0  & 1 
\end{array}
\right) \enspace . 
\end{equation} 
This tropical polyhedral cone is represented 
on the left hand side of Figure~\ref{FigDeduc2} by the closed region 
in dark gray together with the line segments joining the points $G_{1 \cdot }$, 
$G_{2 \cdot }$ and $G_{3 \cdot }$ to it. This illustration is in barycentric 
coordinates, meaning that a vector $(x_1,x_2,x_3)$ of $\maxplus^3$ 
is represented by the barycenter with weights $(e^{x_1}, e^{x_2}, e^{x_3})$ 
of the three vertices of a simplex. Then, two vectors that are proportional 
in the tropical sense are represented by the same point. This is convenient
to make two-dimensional pictures of tropical convex cones of dimension three.
Moreover, the barycentric representation permits to show
vectors with infinite entries: the latter appear at the 
boundary of the simplex. 

\begin{figure}
\renewcommand{\baryx}{x_1}
\renewcommand{\baryy}{x_2}
\renewcommand{\baryz}{x_3}
\colorlet{mygreen}{gray!50!black}

\begin{center}
\begin{minipage}[b]{0.32\textwidth}
\begin{tikzpicture}%
[scale=0.65,>=triangle 45
,vtx/.style={mygreen},
ray/.style={myred}]
\equilateral{7}{100};

\barycenter{g1}{\expo{-3}}{\expo{0}}{\expo{0}};
\barycenter{g2}{\expo{0}}{\expo{-3}}{\expo{0}};
\barycenter{g3}{\expo{0}}{\expo{0}}{\expo{-3}};
\barycenter{h1}{\expo{1}}{\expo{0}}{\expo{0}};
\barycenter{h2}{\expo{0}}{\expo{1}}{\expo{0}};
\barycenter{h3}{\expo{0}}{\expo{0}}{\expo{1}};
\barycenter{f1}{\expo{-1}}{\expo{0}}{\expo{0}};
\barycenter{f2}{\expo{0}}{\expo{-1}}{\expo{0}};
\barycenter{f3}{\expo{0}}{\expo{0}}{\expo{-1}};

\barycenter{e1}{\expo{0}}{0}{0};
\barycenter{e2}{0}{\expo{0}}{0};
\barycenter{e3}{0}{0}{\expo{0}};

\filldraw[gray,draw=black,opacity=0.9,very thick] (g1) -- (f1) -- (h3) -- (f2) -- (g2) -- (f2) -- (h1) -- (f3) -- (g3) -- (f3) -- (h2) -- (f1) -- cycle;

\filldraw[vtx] (g1) circle (0.75ex) node[below] {$G_{1\cdot }$};
\filldraw[vtx] (g2) circle (0.75ex) node[below] {$G_{2\cdot }$};
\filldraw[vtx] (g3) circle (0.75ex) node[below] {$G_{3\cdot }$};

\filldraw[vtx] (h1) circle (0.75ex) node[below] {$G_{4\cdot }$};
\filldraw[vtx] (h2) circle (0.75ex) node[below] {$G_{5\cdot }$};
\filldraw[vtx] (h3) circle (0.75ex) node[above] {$G_{6\cdot }$};

\end{tikzpicture}
\end{minipage}
\begin{minipage}[b]{0.32\textwidth}
\begin{tikzpicture}%
[scale=0.65,>=triangle 45
,vtx/.style={mygreen},
ray/.style={myred}]
\equilateral{7}{100};

\barycenter{g1}{\expo{-3}}{\expo{0}}{\expo{0}};
\barycenter{g2}{\expo{0}}{\expo{-3}}{\expo{0}};
\barycenter{g3}{\expo{0}}{\expo{0}}{\expo{-3}};
\barycenter{h1}{\expo{1}}{\expo{0}}{\expo{0}};
\barycenter{h2}{\expo{0}}{\expo{1}}{\expo{0}};
\barycenter{h3}{\expo{0}}{\expo{0}}{\expo{1}};
\barycenter{f1}{\expo{-1}}{\expo{0}}{\expo{0}};
\barycenter{f2}{\expo{0}}{\expo{-1}}{\expo{0}};
\barycenter{f3}{\expo{0}}{\expo{0}}{\expo{-1}};

\barycenter{e1}{\expo{0}}{0}{0};
\barycenter{e2}{0}{\expo{0}}{0};
\barycenter{e3}{0}{0}{\expo{0}};

\filldraw[gray,draw=black,opacity=0.9,thick] (g1) -- (f1) -- (h3) -- (f2) -- (g2) -- (f2) -- (h1) -- (f3) -- (g3) -- (f3) -- (h2) -- (f1) -- cycle;

\filldraw[gray,draw=black,opacity=0.6,very thick] (g1) -- (e3) -- (e1) -- (g3) -- (f3) -- (h2) -- (f1) -- cycle;

\end{tikzpicture}
\end{minipage}
\begin{minipage}[b]{0.32\textwidth}
\begin{tikzpicture}%
[scale=0.65,>=triangle 45
,vtx/.style={mygreen},
ray/.style={myred}]
\equilateral{7}{100};

\barycenter{g1}{\expo{-3}}{\expo{0}}{\expo{0}};
\barycenter{g2}{\expo{0}}{\expo{-3}}{\expo{0}};
\barycenter{g3}{\expo{0}}{\expo{0}}{\expo{-3}};
\barycenter{h1}{\expo{1}}{\expo{0}}{\expo{0}};
\barycenter{h2}{\expo{0}}{\expo{1}}{\expo{0}};
\barycenter{h3}{\expo{0}}{\expo{0}}{\expo{1}};
\barycenter{f1}{\expo{-1}}{\expo{0}}{\expo{0}};
\barycenter{f2}{\expo{0}}{\expo{-1}}{\expo{0}};
\barycenter{f3}{\expo{0}}{\expo{0}}{\expo{-1}};
\barycenter{k1}{0}{\expo{3}}{\expo{0}};
\barycenter{k2}{\expo{0}}{\expo{3}}{0};

\barycenter{l1}{\expo{0}}{\expo{1}}{0};
\barycenter{l2}{0}{\expo{1}}{\expo{0}};

\barycenter{i1}{0}{\expo{0}}{\expo{0}};
\barycenter{i3}{\expo{0}}{\expo{0}}{0};

\barycenter{e1}{\expo{0}}{0}{0};
\barycenter{e2}{0}{\expo{0}}{0};
\barycenter{e3}{0}{0}{\expo{0}};

\filldraw[gray,draw=black,opacity=0.9] (g1) -- (f1) -- (h3) -- (f2) -- (g2) -- (f2) -- (h1) -- (f3) -- (g3) -- (f3) -- (h2) -- (f1) -- cycle;

\filldraw[lightgray,draw=black,opacity=0.9,thick] (e1) -- (k1);
\filldraw[lightgray,draw=black,opacity=0.2,very thin] (e1) -- (k1) -- (e3) -- cycle;

\filldraw[lightgray,draw=black,opacity=0.9,thick] (e3) -- (k2);
\filldraw[lightgray,draw=black,opacity=0.2,very thin] (e3) -- (k2) -- (e1) -- cycle;

\filldraw[lightgray,draw=black,opacity=0.9,thick] (i3) -- (f3);
\filldraw[lightgray,draw=black,opacity=0.9,thick] (f3) -- (l2);

\filldraw[lightgray,draw=black,opacity=0.9,thick] (i1) -- (f1);
\filldraw[lightgray,draw=black,opacity=0.9,thick] (f1) -- (l1);
\end{tikzpicture}
\end{minipage}

\end{center}

\caption{Illustration of the proof of Proposition~\ref{CorEquivExtPolars}. 
(Left): The row space $\RowSpace{G}$ of the matrix $G$ in~\eqref{MatrixG}. 
(Middle): The tropical cone $\sK$ obtained by adding to the 
generators of $\RowSpace{G}$ the vectors ${\rm e}^1$ and ${\rm e}^3$ 
of the canonical basis of $\maxplus^3$. 
(Right): The four extreme inequalities of type $2$.}
\label{FigDeduc2}
\end{figure}
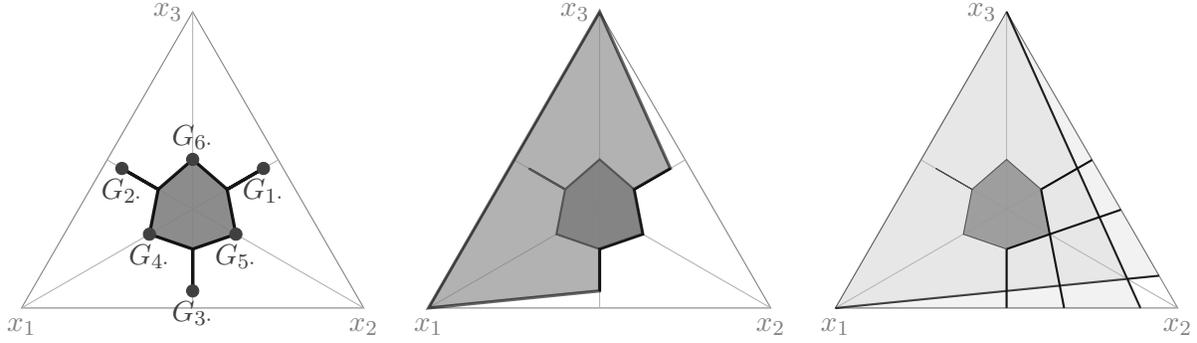
The extreme rays of the $2$nd polar of $\RowSpace{G}$ 
correspond to the following inequalities:
\[
\begin{array}{rlccrl}
 x_1 \oplus (-1)x_3 &\geq (-1)x_2 &&&
 x_1 &\geq (-3)x_2\\
 x_3 &\geq (-3)x_2 &&&
(-1)x_1 \oplus x_3 &\geq (-1)x_2\\
 x_1 &\geq \mpzero &&&
 x_3 &\geq \mpzero
\end{array}
\]
The tropical cones associated with the four extreme inequalities of 
type 2 above are represented on the right hand side of Figure~\ref{FigDeduc2}. 
According to the proof of Proposition~\ref{CorEquivExtPolars}, 
these four non-trivial inequalities precisely define the cone generated by 
the rows of $G$ and the vectors ${\rm e}^1$ and ${\rm e}^3$, 
as illustrated in the middle part of Figure~\ref{FigDeduc2}.
\end{example}

The extreme vectors of tropical polyhedral cones 
can be characterized combinatorially in terms of \NEW{directed} hypergraphs 
(we shall also use undirected hypergraphs in the sequel, 
so we shall always make it clear whether the hypergraph is directed), 
see~\cite{AGG09,AGG10,AllamigeonThesis}. 
Let us recall that a directed hypergraph is a couple $(N,E)$, 
where $N$ is a set of nodes and $E$ is a set of directed hyperedges, 
which are of the form $(M,M')$ with $M,M'\subseteq N$. 
When each of the sets $M$ and $M'$ has only one element, 
we say that the hyperedge is an \NEW{edge}. 
The notion of reachability in directed hypergraphs can be defined recursively: 
node $r$ is said to be \NEW{reachable} from node $h$ if $r=h$ or there exists 
a directed hyperedge $(M,M')\in E$ such that $r\in M'$ and any node in 
$M$ is reachable from $h$. As usual, 
a strongly connected component is a maximal subset of 
nodes $C$ satisfying the property that every node in $C$ is reachable 
from any node in $C$.  

Given a tropical cone $\sK \subseteq \maxplus^n$ defined by $p$ linear inequalities 
$a^r  x \leq b^r x$, where $\{a^r\}_{r\in[p]}$ and $\{b^r\}_{r\in[p]}$ are two
families of vectors of $\maxplus^n$,
the \NEW{tangent directed hypergraph} of $\sK$ at $y\in \maxplus^n$ is defined as the 
directed hypergraph $\mathcal{H}(\sK,y)=(N,E)$ with set of nodes  
$N=\left\{i\in [n]\mid y_i\neq \mpzero \right\}$ 
and set of hyperedges 
\[
E= \left\{(\argmax(b^r y),\argmax(a^r y)) \mid 
r\in [p], a^r y = b^r y \neq \mpzero \right\}\enspace ,
\]
where $\argmax(c x)$ is the argument of the maximum 
$c x = \oplus_{i \in [n]} c_i x_i= \max_{i \in [n]} (c_i + x_i)$ 
for any $c,x\in \maxplus^n$. 

In~\cite{AllamigeonThesis},\cite[Theorem~3.7]{AGG10}, it was shown that a vector $y$ 
of a tropical polyhedral cone $\sK$ is extreme if, and only if, 
the tangent directed hypergraph $\mathcal{H}(\sK , y)$ of $\sK$ at $y$ 
admits a smallest strongly connected component. The term
``smallest'' refers to the order relation in which a strongly connected
component $C_1$ is smaller than a strongly connected component $C_2$
if $C_1$ has a node which is reachable from a node of $C_2$. 
So, this condition means that there is a node which is reachable
for all the other nodes in the tangent directed hypergraph $\mathcal{H}(\sK , y)$. We refer the reader to~\cite{AGG10} and to~\cite{AllamigeonThesis} for details
(note that in the latter reference, the order 
of strongly connected components is reversed).

In the special case of the $i$th polar of $\RowSpace{G}$, 
a simpler characterization of its extreme vectors holds. 
The following theorem shows that a   
vector $b\in \ipolar{\RowSpace{G}}$ with $b_i\neq \mpzero $ 
is extreme if, and only if, the tangent directed hypergraph 
$\mathcal{H}(\ipolar{\RowSpace{G}},b)$ of $\ipolar{\RowSpace{G}}$ at $b$ 
contains a star-like directed sub-hypergraph which is reduced to edges directed to node $i$ and leaving the other nodes, see Figure~\ref{FigureStar}. 
Alternative characterizations are given in Theorem~5 of~\cite{GK09}
and Proposition~5.13 of~\cite{AllamigeonThesis}.

\begin{figure}
\begin{center}
\input{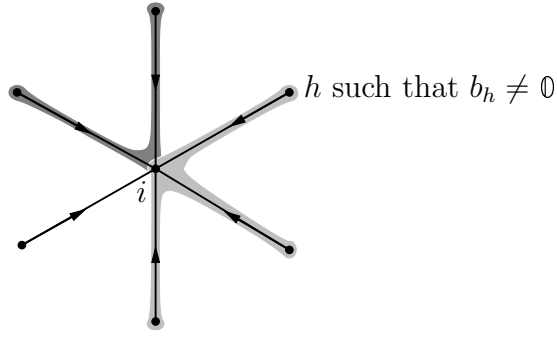}
\end{center}
\caption{Tangent directed hypergraph at an extreme vector $b$ of the $i$th polar of a tropical polyhedral cone. The edges of the star-like directed sub-hypergraph are in black (thin lines). The whole directed hypergraph may contain additional hyperedges pointing to node $i$, here in gray.}
\label{FigureStar}
 \end{figure}

\begin{theorem}[Star-like directed sub-hypergraph]\label{TheoHyperExtremeCharac} Let $b$ be a vector of the $i$th polar of the row space of $G$, i.e.\ $b\in \ipolar{\RowSpace{G}}$, such that $b_i\neq \mpzero$. Then, $b$ belongs to an extreme ray of this $i$th polar if, and only if, the tangent directed hypergraph $\mathcal{H}(\ipolar{\RowSpace{G}},b)$ of $\ipolar{\RowSpace{G}}$ at $b$ 
contains all the edges $(\{h\},\{i\})$ for $h\in \left\{j\in\setexcept{n}{i} \mid b_j\neq \mpzero \right\}$.  
\end{theorem}

\begin{proof}
In the first place, assume that $\mathcal{H}(\ipolar{\RowSpace{G}},b)$ does not contain the edge 
$(\{h\},\{i\})$ for some $h\in \left\{j\in\setexcept{n}{i} \mid b_j\neq \mpzero \right\}$. 
Then, for each $r\in [p]$ such that $b_i G_{ri}= \oplus_{j\in\setexcept{n}{i}} b_j G_{rj}\neq \mpzero $, 
we have $\argmax(\oplus_{j\in\setexcept{n}{i}} b_j G_{rj}) \neq \{ h \}$. 
Therefore, there exists $\lambda < b_h$ such that $b'\in \ipolar{\RowSpace{G}}$, 
where the vector $b'\in \maxplus^n$ is defined by 
$b'_h=\lambda $ and $b'_j=b_j$ for all $j\in\setexcept{n}{h}$. 
Since $\uvector^h$ also belongs to $\ipolar{\RowSpace{G}}$ 
and $b=b'\oplus b_h \uvector^h$, it follows that $b$ is not extreme. 

Conversely, assume now that $\mathcal{H}(\ipolar{\RowSpace{G}},b)$ contains the edge 
$(\{h\},\{i\})$ for each $h\in\setexcept{n}{i}$ such that $b_h\neq \mpzero $. 
Suppose that $b=b'\oplus b''$, where $b',b''\in \ipolar{\RowSpace{G}}$. 
Then, either we have $b_i=b'_i$ or $b_i=b''_i$. Assume, 
without loss of generality, that $b_i=b'_i$. 
We claim that in this case $b=b'$, which proves that $b$ is extreme. 
By the contrary, suppose that $b\neq b'$. Then, 
there exists $h\in \left\{j\in\setexcept{n}{i} \mid b_j\neq \mpzero \right\}$ 
such that $b'_h < b_h$. Since $\mathcal{H}(\ipolar{\RowSpace{G}},b)$ 
contains the edge $(\{h\},\{i\})$, there exists $r\in [p]$ such that 
$b_i G_{ri}= \oplus_{j\in\setexcept{n}{i}} b_j G_{rj}\neq \mpzero $ and 
$\argmax(\oplus_{j\in\setexcept{n}{i}} b_j G_{rj})= \{ h \}$. 
But as $b'\leq b$, $b'_h < b_h$ and $b'_i=b_i$, it follows that 
$b'_i G_{ri} = b_i G_{ri} = b_h G_{rh}> \oplus_{j\in\setexcept{n}{i}} b'_j G_{rj}$, 
contradicting the fact that $b'\in \ipolar{\RowSpace{G}}$. 
This proves our claim. 
\end{proof}

\section{Covering by level sets and hypergraph transversals}

We begin this section by characterizing the extreme inequalities of a 
tropical polyhedral cone in terms of minimal elements of tropical polyhedra. 
Recall that a vector $z$ of a set $\cZ\subseteq \maxplus^{n-1}$ is a 
\NEW{minimal element} of $\cZ$ if $z'\in \cZ$ and $z'\leq z$ imply $z'=z$. 

\begin{theorem}\label{TheorExtremeMinimal}
A system of representatives of the extreme inequalities
satisfied by the elements of the row space $\RowSpace{G}$ 
consists of the trivial inequalities,
together with the inequalities of type $i$, 
$x_i \leq \oplus_{j\in \setexcept{n}{i}} z_j x_j$ where $i\in [n]$, 
in which the vector of coefficients $z$ is a 
minimal element of the tropical polyhedron
\[
\cZ_i \defi \left\{ z=(z_j)_{j\in \setexcept{n}{i}}\in \maxplus^{n-1}\mid \oplus_{j\in\setexcept{n}{i}}z_j G_{\cdot j}\geq G_{\cdot i} \right\} \enspace .
\]
\end{theorem}
\begin{proof}
By Proposition~\ref{PropTypeI}, 
we only need to show the characterization for extreme inequalities of type $i$, 
so let us consider an inequality of type $i$,
\begin{align}\label{e-typei}
x_i \leq \bigoplus_{j\in\setexcept{n}{i}} z_j x_j  \enspace,
\end{align}
in the polar cone $\polar{\RowSpace{G}}$ of $\RowSpace{G}$.
Since every row of $G$ must satisfy this inequality,
we have
\[
G_{\cdot i} \leq \bigoplus_{j\in\setexcept{n}{i}} z_j G_{\cdot j} \enspace .
\]
Assume first that $z$ is a minimal element of $\cZ_i$.
Let us write Inequality~\eqref{e-typei} 
in the form $a x\leq b x$.
If this inequality is not extreme in $\polar{\RowSpace{G}}$,
then, we can write $(a,b)=(a',b')\oplus (a'',b'')$, 
where $(a',b')$ and $(a'',b'')$ belong to $\polar{\RowSpace{G}}$ 
and both of them differ from $(a,b)$. 
We deduce from $\uvector^i=a=a'\oplus a''$
that either $a'=\uvector^i$ or $a''=\uvector^i$, say $a'=\uvector^i$. Then, 
the inequality $a' x\leq b' x$ is of type $i$, 
because $b'_i\leq b_i =\mpzero$ and so $b'_i =\mpzero$. Moreover, 
the vector $z'\in \maxplus^{n-1}$ defined by $z_j'=b'_j$ for $j\in\setexcept{n}{i}$ belongs
to $\cZ_i$. Our assumption that $(a,b)\neq (a',b')$ implies that
$z'$ is (strictly) smaller than $z$,
contradicting the minimality of $z$. It follows that Inequality~\eqref{e-typei} is extreme.

Conversely, if $z$ is not minimal, since the set of elements
in $\cZ_i$ smaller than $z$ is compact, there exists a minimal
element $z'\in \cZ_i$ such that $z'\leq z$. By the definition of $\cZ_i$,
the inequality of type $i$ arising from $z'$ belongs to the polar
of $\RowSpace{G}$, and since $z'\leq z$, Inequality~\eqref{e-typei}
can be obtained by summing to the former inequality suitable multiples
of the inequalities $x_j\geq \mpzero $, $j\in\setexcept{n}{i}$. 
Therefore, we conclude that Inequality~\eqref{e-typei} is not extreme.
\end{proof}

For $i\in [n]$, let   
\[
S_i \defi \{k\in [p]\mid G_{k i}\neq \mpzero \} \enspace . 
\] 
Given a scalar $\lambda\in \maxplus $ and $j\in \setexcept{n}{i}$, 
we define the level set
\[
L_j(\lambda) \defi \{ k\in S_i \mid \lambda G_{k j}\geq G_{k i} \} \enspace .
\]
Then, for any $z\in \maxplus^{n-1}$ we have 
\[
z\in \cZ_i \iff 
\bigoplus_{j\in \setexcept{n}{i}} z_j G_{\cdot j}\geq G_{\cdot i} \iff 
S_i \subseteq \bigcup_{j\in\setexcept{n}{i}} L_j(z_j) \enspace .
\]
Since the maps $\lambda \mapsto \lambda G_{k j}$ are non-decreasing, 
it follows that the family of level sets $\{L_j(\lambda)\}_{\lambda\in \maxplus }$ is a chain. 
For each $j\in \setexcept{n}{i}$, it consists of at most $p+1$ sets, 
say $\emptyset = L_j^1\varsubsetneq L_j^2 \cdots \varsubsetneq L_j^{k_j}$ for some $k_j\in [p+1]$. 
Note that for each level set $L_j^{r}$ there exists a 
minimal $\lambda \in \maxplus $ such that $L_j(\lambda ) = L_j^r$, 
which is given by 
\[
w_j^r \defi \max \left\{G_{k j}^{-1}G_{k i} \mid k\in L_j^r \right\} \enspace , 
\] 
where the maximum over the empty set is defined to be $\mpzero $, and 
$G_{k j}^{-1}$ is understood in the tropical sense (i.e., $-G_{kj}$ with
the usual notation).

We shall say that $\{L_j^{r_j}\}_{j\in \setexcept{n}{i}}$, 
where $r_j \in [k_j]$ for $j\in \setexcept{n}{i}$,  
is a minimal cover of $S_i$ if 
$S_i \subseteq \cup_{j\in \setexcept{n}{i}}L_j^{r_j}$
but this inclusion is no longer satisfied if some non-empty set 
$L_j^{r_j}$ is replaced by $L_j^{r_j-1}$. 

The following simple observation gives a combinatorial interpretation
of the extreme rays of the polar in terms of set covers, which we shall
relate to minimal transversals in hypergraphs at the end of this section.

\begin{proposition}\label{prop-corres}
The minimal elements of $\cZ_i$ correspond to the minimal covers of $S_i$ 
by level sets $L_j^{r_j}$ as $j\in \setexcept{n}{i}$ and $r_j \in [k_j]$.
\end{proposition}

\begin{proof}
In the first place, assume that $z$ is a minimal element of $\cZ_i$. 
For each $j\in\setexcept{n}{i}$, let $r_j\in [k_j]$ be 
such that $L_j^{r_j}=L_j(z_j)$. Then, 
since $\oplus_{j\in \setexcept{n}{i}} z_j G_{\cdot j}\geq G_{\cdot i}$, 
we have 
\[
S_i \subseteq \bigcup_{j\in\setexcept{n}{i}} L_j(z_j) = \bigcup_{j\in\setexcept{n}{i}} L_j^{r_j} \enspace , 
\]
showing that the family of level sets 
$\{L_j^{r_j}\}_{j\in\setexcept{n}{i}}$  
is a cover of $S_i$. We claim that this cover must be minimal. 
Otherwise, in $\{L_j^{r_j}\}_{j\in\setexcept{n}{i}}$ we could replace some non-empty set $L_h^{r_h}$ by $L_h^{r_h-1}$ 
and still obtain a cover of $S_i$. Therefore, if we replaced the component $z_h$ of $z$ by $w_h^{r_h-1}$, 
we would obtain an element of $\cZ_i$. However, since $z_h \geq w_h^{r_h} > w_h^{r_h-1}$, 
this would contradict the minimality of $z$ 
(indeed, as $z$ is minimal, note that 
we must have $z_j=w_j^{r_j}$ for $j\in \setexcept{n}{i}$), 
proving our claim.   

Conversely, assume that $\{L_j^{r_j}\}_{j\in\setexcept{n}{i}}$ is a minimal cover of $S_i$. 
If we define $z\in \maxplus^{n-1}$ by $z_j=w_j^{r_j}$ for $j\in \setexcept{n}{i}$,   
since $L_j(z_j)=L_j(w_j^{r_j})=L_j^{r_j}$, it follows that   
\[
S_i\subseteq \bigcup_{j\in\setexcept{n}{i}} L_j^{r_j}=\bigcup_{j\in\setexcept{n}{i}} L_j(z_j) \enspace ,
\]
and thus $z\in \cZ_i$. If $z$ was not minimal, 
we could decrease one of its components, say $z_h$, 
so that the resulting vector still belongs to $\cZ_i$. 
This would mean that in $\{L_j^{r_j}\}_{j\in \setexcept{n}{i}}$ 
we could replace the set $L_h^{r_h}$ by $L_h^{r_h-1}$ 
and still obtain a cover of $S_i$. However, 
this contradicts that $\{L_j^{r_j}\}_{j\in \setexcept{n}{i}}$ is a minimal cover of $S_i$.
\end{proof}

We now analyze two examples of tropical cones 
with $p$ generators in dimension $n$.
The first one shows that the growth
of the number of extreme rays of the polar of such a cone cannot be polynomially bounded in $p$ and $n$.

\begin{example}\label{Example1}
Assume that $p$ divides $n-1$.
Then, we can choose the matrix $G$
in such a way that the polar of the cone $\RowSpace{G}$ generated by its rows
has at least
\[
((n-1)/p)^p
\]
extreme rays.

Let $q=(n-1)/p$ and $G=[F_1,\ldots,F_p,e]$,
where $e$ is the $p$ dimensional tropical unit column vector 
(with all entries equal to $\mpone $),
and for $k\in [p]$, $F_k$ is a $p\times q$ matrix 
such that the maximum in each column is attained on
row $k$ and only on this row. 
An example of such a matrix in which $n=7$, $p=3$ and $q=2$ is:
\[
G=\left(\begin{array}{cc|cc|cc|c}
5 & 4 & 1 & 1 & 1 & 1 & 0 \\
1 & 1 & 5 & 4 & 1 & 1 & 0 \\
1 & 1 & 1 & 1 & 5 & 4 & 0
\end{array}\right) \enspace . 
\] 

We obtain a minimal cover of $S_n=[p]$ by level
sets as follows. For each $k\in [p]$, select precisely
one index $j_k$ in the set of column indices of $F_k$.
Set $z_{j_k}=G_{k j_k}^{-1}$. Since the maximum of column $j_k$
is attained only on row $k$, we have $L_{j_k}(z_{j_k})=\{k\}$. 
If we define $z_{j}=\mpzero$ for $j\not \in \{j_k\mid k\in [p]\}$, 
it follows that $\{ L_{j}(z_{j})\}_{j\in [n-1]} $ is a minimal
cover of $S_n$ by level sets. 
By Proposition~\ref{prop-corres}, each of
these minimal covers yield a minimal point of $\cZ_n$, and so,
an extreme vector of the polar of $\RowSpace{G}$. 

Since for each $k\in [p]$ there are $q$ ways to choose $j_k$, 
there is a total number of $q^p=((n-1)/p)^p$ choices. 
Note that each of these choices leads to a different extreme vector. 
\end{example}

\begin{example}[Cyclic polyhedral cone]\label{Example2}
Consider the $p\times n$ matrix $G$ defined by $G_{ij}=t_i^{j-1}$, 
where $t_1<t_2<\cdots <t_p$ are $p$ real numbers.
The classical convex cone generated by the rows of $G$ is the \NEW{cyclic polyhedral cone}. Among all the cones in dimension $n$ generated by $p$ vectors, the latter is known to maximize the number of facets, or equivalently, the number
of extreme rays of its polar. This result is part of the celebrated McMullen upper bound theorem~\cite{mcmullen70}, see~\cite{ziegler98,matousek} for more background.  The cyclic polyhedral cone can be defined in the same way in the tropical case. Thus, the exponentiation is now understood tropically, so that the entries of $G$ are given by $G_{ij}=t_i\times (j-1)$, and the \NEW{tropical cyclic polyhedral cone}, $\cP(p,n)$, is the set of all tropical linear combinations of the rows of $G$. This cone was first considered by Block and Yu in~\cite{blockyu06},
and a dual notion (polars of cyclic polyhedral cones) depending
on a sign pattern was studied in~\cite{AGK09}, see below.  
\end{example}

The number of extreme rays of the classical polar of the cyclic
polyhedral cone is known to be of order $p^{\floor{(n-1)/2}}$ as $p$
tends to infinity, for a fixed $n$, see for example~\cite{matousek}. 
This should be opposed to the tropical case, in which
the number of extreme rays  is polynomially bounded, as
shown by the following proposition.
\begin{proposition}\label{PropExtCyclic} 
The polar of the tropical cyclic polyhedral cone $\cP(p,n)$ is generated by 
\[ 2n+\sum_{i=1}^{n}\left( (p-1)(i-1)(n-i)+n-1 \right) = O\left( pn^3 \right)
\]
extreme vectors.
\end{proposition}

\begin{proof} 
In order to prove this proposition, in the first place it is necessary 
to recall some concepts and results from~\cite{AGK09}. 

A \NEW{sign pattern} for the cyclic polyhedral cone $\cP(p,n)$ is a 
$p\times n$ matrix $(\epsilon_{ij})$ 
whose entries are $+$ and $-$ signs. 
The polar of the signed cyclic polyhedral cone with sign pattern 
$(\epsilon_{ij})$ is defined (see~\cite[Definition~1]{AGK09}) 
as the set of vectors $x\in \maxplus^n$ such that 
\[
a^r x \leq b^r x \; \makebox{ for all } \; r\in [p] \; ,
\] 
where $a^r_j=G_{rj}$ if $\epsilon_{rj}$ is $-$ 
(resp. $b^r_j=G_{rj}$ if $\epsilon_{rj}$ is $+$) 
and $a^r_j=\mpzero $ otherwise (resp. $b^r_j=\mpzero $ otherwise). 
An \NEW{oriented lattice path} in the sign pattern $(\epsilon_{ij})$ is a  
sequence of entries of $(\epsilon_{ij})$
starting from some top entry $(1,i)$ and ending with some bottom entry $(p,j)$ 
such that the successive entries are always either immediately at the right
or immediately at the bottom of the previous ones. Therefore, 
such a path is composed of vertical segments oriented downward 
(vertical segments are supposed to be composed of at least two entries, 
with exception of the first and last vertical segments which 
are allowed to be composed of only one entry so that every 
paths starts and ends with a vertical segment) 
and horizontal segments oriented to the right 
(which are supposed to be composed of at least two entries). 
An oriented lattice path is said to be \NEW{tropically allowed} for 
the sign pattern $(\epsilon_{ij})$ if the following conditions are satisfied: 

{\em(i)}\/ every sign occurring on the initial vertical segment, 
except possibly the sign at the bottom of the segment, is positive;

{\em(ii)}\/ every sign occurring on the final vertical segment, 
except possibly the sign at the top of the segment, is positive;

{\em(iii)}\/ every sign occurring in any other vertical segment, 
except possibly the signs at the top and bottom of this segment, is positive;

{\em(iv)}\/ for every horizontal segment, 
the pair of signs consisting of the signs of the leftmost and rightmost 
positions of the segment is of the form $(+,-)$ or $(-,+)$; 

{\em(v)}\/ as soon as a pair $(-,+)$ occurs as the pair of 
extreme signs of some horizontal segment, 
the pairs of signs corresponding to all the horizontal segments 
below this one must also be equal to $(-,+)$. 

\begin{figure}
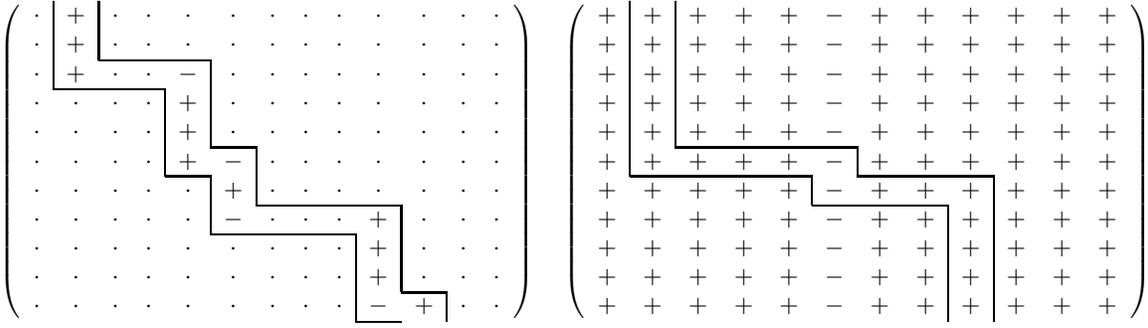

\begin{center}
{\Small
\[
\left( 
\begin{array}{ccccccccccccc}
\cdot&\multicolumn{1}{|c|}{\mop }&\cdot&\cdot&\cdot&\cdot&\cdot&\cdot&\cdot&\cdot&\cdot&\cdot&\cdot \\ 
\cdot&\multicolumn{1}{|c|}{\mop }&\cdot&\cdot&\cdot&\cdot&\cdot&\cdot&\cdot&\cdot&\cdot&\cdot&\cdot \\ \cline{3-5}
\cdot& \multicolumn{1}{|c}{\mop } &\cdot &\cdot &\multicolumn{1}{c|}{\mon}&\cdot&\cdot&\cdot&\cdot&\cdot&\cdot&\cdot&\cdot \\ \cline{2-4}
\cdot&\cdot&\cdot&\cdot&\multicolumn{1}{|c|}{\mop }&\cdot&\cdot&\cdot&\cdot&\cdot&\cdot&\cdot&\cdot \\
\cdot&\cdot&\cdot&\cdot&\multicolumn{1}{|c|}{\mop }&\cdot&\cdot&\cdot&\cdot&\cdot&\cdot&\cdot&\cdot \\ \cline{6-6}
\cdot&\cdot&\cdot&\cdot&\multicolumn{1}{|c}{\mop }&\multicolumn{1}{c|}{\mon }&\cdot&\cdot&\cdot&\cdot&\cdot&\cdot&\cdot  \\ \cline{5-5}
\cdot&\cdot&\cdot&\cdot&\cdot&\multicolumn{1}{|c|}{\mop }&\cdot&\cdot&\cdot&\cdot&\cdot&\cdot&\cdot \\ \cline{7-10}
\cdot&\cdot&\cdot&\cdot&\cdot&\multicolumn{1}{|c}{\mon }&\cdot &\cdot &\cdot &\multicolumn{1}{c|}{\mop }&\cdot&\cdot&\cdot \\ \cline{6-9}
\cdot&\cdot&\cdot&\cdot&\cdot&\cdot&\cdot&\cdot&\cdot&\multicolumn{1}{|c|}{\mop }&\cdot&\cdot&\cdot \\
\cdot&\cdot&\cdot&\cdot&\cdot&\cdot&\cdot&\cdot&\cdot&\multicolumn{1}{|c|}{\mop }&\cdot&\cdot&\cdot \\ \cline{11-11}
\cdot&\cdot&\cdot&\cdot&\cdot&\cdot&\cdot&\cdot&\cdot&\multicolumn{1}{|c}{\mon }&\multicolumn{1}{c|}{\mop }&\cdot&\cdot \\ \cline{10-10}
\end{array}
\right)
\quad 
\left( 
\begin{array}{cccccccccccc}
\mop&\multicolumn{1}{|c|}{\mop }&\mop&\mop&\mop&\mon&\mop&\mop&\mop&\mop&\mop&\mop \\ 
\mop&\multicolumn{1}{|c|}{\mop }&\mop&\mop&\mop&\mon&\mop&\mop&\mop&\mop&\mop&\mop \\ 
\mop&\multicolumn{1}{|c|}{\mop }&\mop&\mop&\mop&\mon&\mop&\mop&\mop&\mop&\mop&\mop \\ 
\mop&\multicolumn{1}{|c|}{\mop }&\mop&\mop&\mop&\mon&\mop&\mop&\mop&\mop&\mop&\mop \\
\mop&\multicolumn{1}{|c|}{\mop }&\mop&\mop&\mop&\mon&\mop&\mop&\mop&\mop&\mop&\mop \\ \cline{3-6}
\mop&\multicolumn{1}{|c}{\mop }&\mop&\mop&\mop&\multicolumn{1}{c|}{\mon }&\mop&\mop&\mop&\mop&\mop&\mop \\ \cline{2-5}\cline{7-9}
\mop&\mop&\mop&\mop&\mop&\multicolumn{1}{|c}{\mon }&\mop&\mop&\multicolumn{1}{c|}{\mop }&\mop&\mop&\mop \\ \cline{6-8}
\mop&\mop&\mop&\mop&\mop&\mon&\mop&\mop&\multicolumn{1}{|c|}{\mop }&\mop&\mop&\mop \\ 
\mop&\mop&\mop&\mop&\mop&\mon&\mop&\mop&\multicolumn{1}{|c|}{\mop }&\mop&\mop&\mop \\
\mop&\mop&\mop&\mop&\mop&\mon&\mop&\mop&\multicolumn{1}{|c|}{\mop }&\mop&\mop&\mop \\
\mop&\mop&\mop&\mop&\mop&\mon&\mop&\mop&\multicolumn{1}{|c|}{\mop }&\mop&\mop&\mop 
\end{array}
\right)
\]}
\end{center}
\caption{Tropically allowed lattice paths. (Left): 
For the given tropically allowed lattice path, 
the signs of the entries indicated by the symbol ``$\cdot $'' are irrelevant. 
(Right): A tropically allowed lattice path with two horizontal segments for 
the sign pattern associated with $i$th polar of $\cP(p,n)$, 
where here $n=12$, $p=11$ and $i=6$.}\label{fig:TropAllPaths}
\end{figure}

Examples of tropically allowed lattice paths are shown 
in Figure~\ref{fig:TropAllPaths}, 
we refer the reader to~\cite{AGK09} for more information.  

The interest in tropically allowed lattice 
paths is given by Theorem~2 of~\cite{AGK09},  
which shows that  the extreme rays of 
the polar of the signed cyclic polyhedral cone 
with sign pattern $(\epsilon_{ij})$ 
are in one to one correspondence with tropically
allowed lattice paths for $(\epsilon_{ij})$. 

Consider now the $i$th polar $\ipolar{\sK}$ of $\cP(p,n)$, 
which consists of the vectors $b\in \maxplus^n$ such that
\[
b_i x_i \leq \bigoplus_{j\in\setexcept{n}{i}} b_j x_j
\]
for all $x\in \cP(p,n)$. By the discussion in Section~\ref{SecExtremePolar},
it follows that the extreme rays of the polar of $\cP(p,n)$, 
with exception of those associated with the trivial inequalities $x_i\geq x_i$, 
correspond to the extreme rays of the cones $\ipolar{\sK}$, $i\in [n]$.
Observe that, if we consider the sign pattern $(\epsilon_{ij})$ all whose 
entries are $+$ signs with exception of column $i$ in which they are $-$ signs, 
then $\ipolar{\sK}$ is the polar of the signed cyclic polyhedral cone 
with sign pattern $(\epsilon_{ij})$. 
Given the structure of this sign pattern, 
its tropically allowed lattice paths 
(which therefore correspond to the extreme inequalities of $\cP(p,n)$ 
associated with the extreme rays of $\ipolar{\sK}$)  
can be classified as follows.

{\em(i)}\/ There are $n-1$ vertical tropically allowed paths,
corresponding to the trivial inequalities $x_j\geq \mpzero $
for $j\in \setexcept{n}{i}$. 

{\em(ii)}\/ There are $n-1$ tropically allowed paths 
with exactly one horizontal segment. 
If the extreme pair of signs of this segment is $(+,-)$, 
then, the $-$ sign must be on the last row, 
and so there are $i-1$ choices for the column containing the $+$ sign 
(i.e. the column containing the first vertical segment).
If the extreme pair of signs of the horizontal segment is $(-,+)$, 
then, the $-$ sign must be on the first row and thus there are $n-i$ 
choices for the column containing the $+$ sign 
(i.e. the column containing the last vertical segment). 
This makes a total of $n-1$ extreme rays, 
corresponding to inequalities of the form 
$t_p^i x_j \geq t_p^ j x_i$ for $j\in [i-1]$ and 
$t_1^i x_j \geq t_1^j x_i$ for $j\in [i+1,n]$,
where for all $r\in [n]$, we set $\interval[r,n]:=\{r,r+1,\dots ,n\}$. 
 
{\em(iii)}\/ The other tropically allowed paths must consist of two horizontal segments
on consecutive rows, which have $(+,-)$ and $(-,+)$
as successive pairs of extreme signs, see Figure~\ref{fig:TropAllPaths} 
for an example. The first row $m$ can take $p-1$ values, and for each of these, we must choose a column $j$ in $[i-1]$ containing 
the first vertical segment of the path and a column $k$ in $[i+1,n]$ 
containing the last vertical segment of the path. 
The associated extreme ray can be shown (see~\cite{AGK09}) 
to correspond to the inequality
\[
t_{m+1}^i t_m^j x_k \oplus t_m^i t_{m+1}^k x_j \geq t_m^j t_{m+1}^k x_i
\enspace .
\]

Thus, for a fixed $i$, we have a total of
\[
n-1+(p-1)(n-i)(i-1) 
\]
extreme inequalities, excluding the trivial inequalities $x_j\geq \mpzero $,
$j\in \setexcept{n}{i}$ (since these inequalities arise several times
for different values of $i$, they must be counted separately). 
Summing the latter quantity over $i\in [n]$, and adding the $2n$ trivial
inequalities $x_i \geq x_i$ and $x_i\geq \mpzero $, $i\in [n]$,
we arrive at the formula given in the proposition.
\end{proof}

\begin{remark}
As mentioned earlier, in classical convex geometry, cyclic polyhedral cones are known to 
maximize the number of extreme rays of the polar, among all
cones with $p$ generators in dimension $n$.
By combining Proposition~\ref{PropExtCyclic} and Example~\ref{Example1}, 
we see that the same is not true in the tropical case. It would
be interesting to find the ``maximizing model'' in this case.
\end{remark}

We now derive some algorithmic consequences of Theorem~\ref{TheorExtremeMinimal}. The extreme rays of the polar of a tropical polyhedral cone can be computed
by the tropical double description method~\cite{AGG10}, see also~\cite{AGG09,AllamigeonThesis} for  more information.
The latter is a general method, which determines the extreme
rays of a tropical polyhedral cone defined as the intersection of tropical half-spaces. This method computes a sequence of intermediate polyhedral cones, given by the intersection of successive half-spaces. Its execution time is polynomial in the size of the input and the maximal number of extreme rays of these intermediate cones. However, there are instances in which this number blows up, so that the execution time can be exponential in the size of the input and the output. Theorem~\ref{TheorExtremeMinimal} will allow
us to exploit known algorithms for variants of a classical problem in 
hypergraph theory: finding all minimal transversals. This leads to
an alternative algorithm to compute the polar, which, by comparison
with the tropical double description method, has the advantage
of running in a time which is quasi-polynomial in the size of the
input and the output. 
However, it should be noted that this alternative algorithm can only be applied to the 
intersection of tropical half-spaces defined by inequalities of the same type $i$ 
(the $i$th polar of a tropical polyhedral cone is 
given by the intersection of such half-spaces), 
while the tropical double description method can handle the intersection 
of tropical half-spaces defined by inequalities of different types. 

Let us recall that, given a (undirected) hypergraph with set of nodes $N$ and 
set of hyperedges $E$ (i.e.\ $E$ is a family of subsets of $N$), 
a \NEW{transversal} or \NEW{hitting set} of this hypergraph is a set $T\subseteq N$ such that $M\cap T\neq \emptyset $ for all $M\in E$. 
A transversal $T$ is minimal if no proper subset of $T$ is also a transversal. 
The \NEW{minimal transversal problem} consists in finding all minimal transversals of a given hypergraph. 
We next show that the minimal covers arising in Proposition~\ref{prop-corres} may be thought of as weighted generalizations of hypergraph transversals. 

Consider a hypergraph with set of nodes
$[n-1]$, and let $E=\{M_1 ,\ldots ,M_p \}$ be its set of hyperedges. 
We associate with this hypergraph the $p\times n$ matrix $G=[F,e]$, 
where $e$ is the $p$ dimensional tropical unit column vector and 
$F$ is the $p\times (n-1)$ matrix defined by: 
$F_{ij}=\mpone $ if $j\in M_i$ and $F_{ij}=\mpzero $ otherwise. 

Then, it can be easily checked that the entries of any minimal 
element of the tropical polyhedron
\[
\left\{z\in \maxplus^{n-1} \mid \oplus_{j\in [n-1]}z_j G_{\cdot j}\geq G_{\cdot n} \right\} = 
\left\{z\in \maxplus^{n-1} \mid F z \geq e \right\}
\]
can only take values in the set $\{ \mpone ,\mpzero \}$ and that 
$z\in \{ \mpone ,\mpzero \}^{n-1}$ is a minimal element of this polyhedron if, and only if, 
the set $\{j\in [n-1] \mid z_j \neq \mpzero \}$ 
is a minimal transversal of the given hypergraph. 

In other words, the rows of $F$ represent the incidence vectors
of hyperedges, and the minimal elements $z$
are the incidence vectors of minimal transversals.

Therefore, by Theorem~\ref{TheorExtremeMinimal}, 
minimal transversals of the given hypergraph correspond to extreme rays of the 
$n$th polar of $\RowSpace{G}$ associated with vectors $b$ such that 
$b_n \neq \mpzero $. 
We summarize this discussion by the following corollary.
\begin{corollary}\label{CorTransversal}
The minimal transversal problem reduces to the computation of 
the extreme rays of the $i$th polar of a tropical cone. \qed
\end{corollary}

Fredman and Khachiyan~\cite{Fredman} showed that the minimal transversal problem can be solved in incremental quasi-polynomial time. 
This means that given a set $S$ of already computed minimal transversals, 
the time needed to compute a new minimal transversal or to decide
that there are no more minimal transversals is bounded by
$2^{\makebox{polylog}(m)}$, where $m=k+|S|$ and $k$ is the size
of the input. Boros, Elbassioni, Gurvich, Khachiyan and Makino extended this result in~\cite{BEGKM02} to the case of systems of monotone linear inequalities, and
considered general dualization problems (see also~\cite{KBEG06}).
Elbassioni showed in~\cite[Theorem~1]{Elbassioni08}, as a consequence of~\cite{BEGKM02}, that the minimal
elements of a set of the form 
\[ 
\left\{ x \in \maxplus^n \mid A x \geq b \enspace , \enspace l\leq x \leq u \right\} \enspace , 
\]
where $b$, $l$ and $u$  are vectors of $\maxplus^n$ and 
$A$ is a $p \times n$ matrix with entries in $\maxplus$, can also be computed 
in incremental quasi-polynomial time. (Actually, the setting of~\cite{Elbassioni08} concerns  ``max-times'' inequalities, but the present setting is equivalent.)

Taking $A$ as the matrix whose columns are the columns of $G$ 
with exception of column $i$ and $b$ as the $i$th column of $G$, 
if we define $l_h\defi \mpzero $ and $u_h \defi w_h^{k_h}$ for 
$h\in\setexcept{n}{i}$, we conclude from Elbassioni's theorem 
that the minimal elements of $\cZ_i$ can be computed in 
incremental quasi-polynomial time. 
Combining this remark with Theorem~\ref{TheorExtremeMinimal}, we obtain:

\begin{corollary}
The extreme rays of the polar of a tropical polyhedral cone
can be computed in incremental quasi-polynomial time.\qed
\end{corollary}

\section{A tropical analogue of Farkas lemma involving mean payoff games}\label{SFarkasGame}

The classical Farkas lemma shows that a (homogeneous) 
linear inequality over the reals
can be logically deduced from a finite family of linear inequalities
if, and only if, it can be expressed as a nonnegative linear combination
of the inequalities in this family. As it was observed in~\cite{katz08}, the same is not true in the tropical setting (see Figure~\ref{FigDeduc} below). This raises the question
of deciding whether 
\begin{align}
A x\leq B x\implies c x\leq d x 
\label{e-todecide}
\end{align}
for all $x \in \maxplus^n$, 
where $A,B$ are $p\times n$ matrices and 
$c,d$ are row vectors of dimension $n$, 
all of them with (effective) entries in $\maxplus $. 
Equivalently, given a finite system of (tropical) linear inequalities, 
we may ask whether one of them is redundant. 
In recent works~\cite{AGG,aggut10}, Akian, Gaubert and Guterman showed
that checking whether a tropical polyhedral cone
is trivial (i.e. reduced to the identically $\mpzero $ vector)
reduces to solving a mean payoff game problem.
We next show that the problem of deciding whether
implication~\eqref{e-todecide} holds also reduces to a mean payoff game problem.
We refer the reader to~\cite{gurvich,zwick} for more background
on these games.

In order to perform this reduction to games, it is convenient 
to establish first some simple preliminary properties. 
In many applications, the finite entries of the matrices $A,B$ and
the vectors $c,d$ are integers. Then, it follows
from the next result that the validity of implication~\eqref{e-todecide}
does not change if one considers real or integer variables. 
In the sequel, if $H$ is a subgroup of $(\R,+)$, 
we denote by $H_{\max}$ the semiring $(H\cup\{-\infty\},\max ,+)$.

\begin{proposition}\label{prop-integerisenough}
Assume that $A,B\in H_{\max }^{p\times n}$ and  
$c,d\in H_{\max }^n$, where $H$ is a subgroup of $(\R,+)$. 
Then, the implication $Ax\leq Bx \implies cx\leq dx$ 
holds for all $x\in \maxplus^n$ if, and only if, 
it holds for all $x\in H_{\max }^n$.
\end{proposition}

\begin{proof}
Let $\sK:=\left\{ x\in \maxplus^n \mid  Ax\leq Bx\right\} $.
The tropical analogue of the Minkowski theorem~\cite{GK06a,GK,BSS} shows that
every vector of $\sK$ is a tropical linear combination of vectors
in the extreme rays of $\sK$. Hence, the implication is valid for
all $x\in \maxplus^n$ if, and only if, 
every vector $x$ in an extreme ray of $\sK$ satisfies $cx\leq dx$.
Since all the vectors in a ray are proportional in the tropical sense, 
it suffices to check that $cx\leq dx$ holds for a suitably normalized 
vector of each extreme ray. For instance, we may normalize a vector
of a ray by requiring that $x_j=0$,
where $j$ is the first index in $[n]$ such that $x_j$ is finite.
The normalized vectors of extreme rays will be referred
to as {\em extreme generators}.
Then, the explicit construction of the extreme generators,
in the tropical double description algorithm~\cite{AGG10}, 
shows that every finite entry of the extreme generators
belongs to the subgroup $H$ of $(\R,+)$, because all the operations
performed by the algorithm preserve this subgroup.
Hence, if the implication holds for all $x\in H_{\max }^n$, 
it holds in particular for all
the normalized vectors of the extreme rays of $\sK$, 
and so it holds for all $x\in \maxplus^n$.
\end{proof}

We shall need the following technical proposition.

\begin{proposition}\label{prop-genex}
The finite entries of every vector $y$ in an extreme ray of the tropical
polyhedral cone $\sK \defi \left\{x\in \maxplus^n\mid Ax\leq Bx\right\}$
satisfy:
\begin{align}\label{e-cbound}
|y_j- y_k|\leq \max_{m\leq n-1 ,\;  r_1,\ldots ,r_{m-1}\in[n]}
M_{j r_1}+M_{r_1 r_2}+\cdots +M_{r_{m-1} k} \enspace , 
\end{align}
where
\[
M_{s t}\defi \max_{i\in [p],\; A_{i s},B_{i t}\neq \mpzero } |A_{i s}-B_{i t}|  
\enspace .
\]
In particular, $|y_j- y_k|\leq M\defi (n-1)\max_{s,t\in [n]} M_{s t}$.
\end{proposition}

\begin{proof}
To bound $|y_j-y_k|$, we shall use the characterization 
of the extreme vectors of $\sK$ in terms of 
tangent directed hypergraphs established in~\cite[Theorem~3.7]{AGG10}, 
which was already recalled in Section~\ref{SecExtremePolar}. 
The tangent directed hypergraph of $\sK$ at 
$y\in \maxplus^n$ is the directed hypergraph $\mathcal{H}(\sK,y)=(N,E)$ 
with set of nodes  
$N=\left\{i\in [n]\mid y_i\neq \mpzero \right\}$ 
and set of hyperedges 
\[
E= \left\{ (\{j\in[n]\mid B_{rj}y_j =B_{r\cdot}y\},
\{j\in[n]\mid A_{r j}y_j =A_{r\cdot}y\}) \mid 
r\in [p], A_{r\cdot} y = B_{r\cdot} y \neq \mpzero \right\} \enspace . 
\]  
Here, we shall also use an undirected graph, denoted by $\mathcal{G}(\sK,y)$,
with the same set of nodes $N$ and with an edge connecting nodes $j$ and $k$ 
if there exists an hyperedge $(M,M')\in E$ such that $j\in M$ and $k\in M'$. 
In other words, the edge $(j,k)$ belongs to $\mathcal{G}(\sK,y)$ if, 
and only if, $A_{rk}y_k = A_{r\cdot} y = B_{r\cdot} y=B_{rj}y_j\neq \mpzero $ 
for some $r\in [p]$.  
Recall that the result of~\cite{AGG10} shows that $y\in \sK$ belongs
to an extreme ray of $\sK$ if, and only if, $\mathcal{H}(\sK,y)$ 
has a smallest strongly connected component. 
It follows that in particular the underlying undirected graph 
$\mathcal{G}(\sK,y)$ must be connected. 
Note that for any edge $(j,k)$ of $\mathcal G(\sK,y)$ we have 
(with the usual notation) $A_{r k}+ y_k=B_{r j}+ y_j$ 
for some $r\in [p]$, and so $|y_j- y_k|\leq M_{k j}$. 
Consider now any two nodes $j,k$ of $N$. 
Since $\mathcal{G}(\sK,y)$ is connected, 
it must contain an undirected path $j,r_1,\ldots,r_{m-1},k$
of length $m\leq |N|-1 \leq n-1$ connecting these two nodes,
which shows~\eqref{e-cbound}. 
\end{proof}

The following immediate corollary shows that when 
implication~\eqref{e-todecide} does not hold, 
we can construct a counter example by assigning
to the variables values which are not too large.

\begin{corollary}\label{coro-smallcex}
If the implication $Ax\leq Bx\implies cx\leq dx$ does not hold, 
there is a vector $y\in \maxplus^n$ that satisfies~\eqref{e-cbound}
such that $A y\leq B y$ and $c y >d y$ (a counter example).
\end{corollary}

\begin{proof}
If implication~\eqref{e-todecide} does not hold, 
there is at least one extreme generator $y$ of the tropical polyhedral cone 
$\left\{x\in \maxplus^n\mid Ax\leq Bx\right\}$ 
such that $c y > d y$, 
and so~\eqref{e-cbound} is valid for $y$.
\end{proof}

\begin{remark}
The previous corollary is related to the ``small model property''
established by Bezem, Nieuwenhuis, and Rodr\'iguez-Carbonell
for the ``max-atom'' problem. Lemma~1 of~\cite{BezemNieuwenhuisCarbonnell10} 
(see also~\cite{BezemNieuwenhuisCarbonnell08}) deals with 
a system of inequalities of the form $x_i\leq k_{i,r,s}+\max(x_r,x_s)$,
for all $(i,r,s) \in V$, where the set $V$ is given, and every 
coefficient $k_{i,r,s}$ is a given integer. They show that if
this system has a finite integer solution, then, it also has
a finite integer solution $y$ such that $|y_i-y_j|\leq \sum_{(i,r,s)\in V}
|k_{i,r,s}|$.
\end{remark}

We now present the reduction to games.
Given a scalar $\lambda\in\R$, we shall consider the system of inequalities 
$A x\leq B x$, $\lambda d x \leq c x$. Following~\cite{AGG}, with this system
we associate a mean payoff game 
in which there are two players, the maximizer ``Max'' and the minimizer ``Min''. 
This game can be represented by a bipartite digraph $\sG_\lambda$ 
with two classes of nodes: the {\em row} nodes $[p+1]$ and the {\em column} nodes $[n]$. 
For $i\in [p]$ and $j\in [n]$, we draw an arc with
weight $B_{ij}$ from row node $i$ to column node $j$ if $B_{ij}\in \R$, 
and we draw an arc with weight $-A_{ij}$ from column node $j$ to row node $i$
if $A_{ij}\in \R$. Similarly, we draw an arc from row node $p+1$ to
column node $j$ with weight $c_j$ if $c_j \in \R$, 
and we draw an arc from column node $j$ to row node $p+1$ with weight $-\lambda-d_j$ if $d_j\in \R$. This is illustrated in Example~\ref{ex-deduce} and Figure~\ref{BipartiteDigraph} (Left) below. 

The mean payoff game associated with this bipartite 
digraph consists of moving a token along its edges. 
When the token is at row node $i$, 
Player Max must select an arc leaving $i$ and 
receives the weight of the arc as a payment from Player Min.
When the token is at column node $j$, Player Min
must select an arc leaving $j$, and pays to Player Max the weight of this arc. 
(We warn the reader that an opposite convention of sign is used in~\cite{AGG}: 
it is assumed there that at each step, the player who makes the moves receives 
the amount indicated on the arc, whereas here, Player Max receives this amount, 
even when Player Min makes the move.)

We shall need the following simple assumption, which guarantees
that each player has at least one action available in every node.

\begin{assumption}\label{assump}
For all $j\in [n]$, $d_j\in \R$ or 
there exists $i\in [p]$ such that $A_{ij}\in \R$. 
There exists $j\in[n]$ such that $c_j\in \R$ and, for all $i\in [p]$, 
there exists $j\in [n]$ such that $B_{ij}\in \R$.
\end{assumption}

Observe that, since we are interested in studying 
the implication $A x\leq B x\implies c x\leq d x$, 
we may always assume that the conditions of Assumption~\ref{assump} hold. 
Indeed, by adding to the $p$ inequalities $Ax\leq Bx$ the 
$n$ trivial inequalities $x_j\leq x_j$, $j\in [n]$, 
we obtain an equivalent implication in which 
the first condition of Assumption~\ref{assump} is satisfied. 
If for some $i\in [p]$ we have $B_{ij}= \mpzero $ for all $j\in [n]$, 
then $A x\leq B x$ implies $x_j=\mpzero$ 
for all $j\in [n]$ such that $A_{ij}\neq \mpzero$. 
Therefore, by eliminating the $i$th inequality and the variables 
$x_j$ for which $A_{ij}\neq \mpzero$, 
we obtain a new implication which is equivalent to the original one. 
By repeating this elimination procedure a finite number of times, 
we eventually arrive at an equivalent implication 
(involving a subset of variables) which satisfies 
the last condition of Assumption~\ref{assump}. 
Finally, observe that we may always assume that $c$ is not the identically 
$\mpzero $ vector because otherwise the implication trivially holds.  
Hence, in the sequel, 
{\em we shall always require the matrices to satisfy Assumption~\ref{assump}} 
without stating it explicitly.  

The dynamic programming operator $g_\lambda $ of the game described above is 
the self-map of $\R^n$ given by
\[
[g_\lambda(x)]_j \defi 
\min\Big(\min_{i\in[p],A_{ij}\in \R}\big(-A_{ij}+\max_{k\in[n]}(B_{ik}+x_k)\big),
-\lambda -d_j+\max_{k\in[n]}(c_k+x_k)\Big)  
\]
if $d_j\in \R$,  
\[ 
[g_\lambda(x)]_j \defi 
\min_{i\in[p],A_{ij}\in \R}\big(-A_{ij}+\max_{k\in[n]}(B_{ik}+x_k)\big)  
\]
otherwise. Observe that in this section, for more readability, 
we come back to the usual notation (instead of the tropical one)
when dealing with dynamic programming operators of games. 
The fact that $g_\lambda$ preserves $\R^n$ follows readily from Assumption~\ref{assump}, 
which implies that the maxima and minima appearing in the previous expressions
only take finite values when $x\in\R^n$ and $\lambda\in\R$. Note also
that $g_\lambda$ has a unique continuous extension to $\maxplus^n$, 
that we will denote by the same symbol $g_\lambda$ ($\maxplus^n$ is equipped 
with the product topology which arises when considering the metric 
$(x,y)\mapsto |e^x-e^y|$ on $\maxplus$). Actually, 
the meaning of the previous expressions giving $[g_\lambda(x)]_j$ 
is unambiguous, even when $x\in \maxplus^n$, and this determines the extension.

Observe that $g_\lambda$ satisfies 
$x\leq y \implies g_\lambda(x)\leq g_\lambda(y)$, i.e.\ $g_\lambda$ is order preserving. 
Moreover, for any scalar $\mu \in \maxplus $ and $x\in \maxplus^n$, 
we have $g_\lambda(\mu +x)=\mu +g_\lambda(x)$, 
so we shall say that $g_\lambda$ commutes with the (usual) addition of a scalar. 

We denote by $\rho(f)$ the (non-linear) \NEW{spectral radius} of a continuous order
preserving self-map $f$ of $\maxplus^n$ that commutes
with the addition of a scalar. 
Recall that $\rho(f)$ is defined as the maximal scalar $\mu $ 
for which there exists a non-identically $\mpzero $ vector $x\in \maxplus^n$ 
(non-linear \NEW{eigenvector})
such that $f(x)=\mu +x$. In other words, 
$\rho(f)$ is the maximal ``additive eigenvalue'' of $f$.
We refer the reader to~\cite{AGG} for more background. 

Since the map $g_\lambda$ preserves $\R^n$, is piecewise affine
and sup-norm nonexpansive, the following limit,
called \NEW{cycle time}, is known to exist~\cite{kohlberg}:
\begin{align}\label{Def-cycletime}
\chi(g_\lambda) \defi 
\lim_{k\to \infty}\frac{g_\lambda^k(0)}{k} \enspace.
\end{align}
Here, $g_\lambda^k$ denotes the $k$th iterate of $g_\lambda$ and  
$0$ is the $n$ dimensional zero vector.
Kohlberg actually shows a stronger result, that there is an invariant half-line
on which $g_\lambda$ acts by translation. It follows
easily from this result that the $j$th entry of $\chi(g_\lambda)$, 
denoted by $\chi_j(g_\lambda)$, 
coincides with the value of the game when the initial state is column node 
$j$ and the payoff of an infinite run is defined as the average payment per 
turn made by Player Min, as in~\cite{liggettlippman}. 
(The equivalence is detailed in~\cite{AGG}.) 

A Collatz-Wielandt type formula (\cite[Lemma~2.8]{AGG}, 
see also~\cite{sgjg04}) shows that
\begin{align}\label{e-weakcw}
\rho(g_\lambda)=\overline{\chi}(g_\lambda) \defi 
\lim_{k\to \infty}\frac{1}{k}\max_{j\in [n]}(g_\lambda^k(0))_j  \enspace ,
\end{align}
and so
\begin{align}\label{e-rho-isamax}
 \rho(g_\lambda)
= \max_{j\in[n]}\chi_j(g_\lambda) 
\end{align}
can be interpreted as the value
of an associated mean payoff game in which Player Max is allowed to select
the initial state $j$, see Proposition~2.11 of~\cite{AGG} for details.
We shall refer to $\rho(g_\lambda)$ as
the \NEW{mean payoff} (value) of the latter game.

The introduction of these mean payoff games 
is motivated by the following propositions.

\begin{proposition}\label{prop-rhoS}
The system of inequalities $A x\leq B x$ does not imply 
the scalar inequality $c x \leq d x$ if, 
and only if, $\max_{j\in[n],c_j\neq \mpzero} \chi_j(g_\lambda) \geq 0$ 
for some $\lambda > 0$.
\end{proposition}

\begin{proof}
There exists a vector $y$ such that $A y\leq B y$ and $c y>d y$ if, and only 
if, there exists a scalar $\lambda>0$ such that the tropical polyhedral cone
\begin{equation}\label{Klambda}
\sK_\lambda\defi \left\{x\in \maxplus^n \mid  A x\leq B x ,\enspace \lambda d x\leq c x \right\} 
\end{equation}
contains a vector $x$ for which $c x\neq \mpzero$, i.e. 
$\sK_\lambda$ contains a vector $x$ satisfying $x_j\neq \mpzero $ 
for some $j\in [n]$ such that $c_j\neq \mpzero $.
Then, the proposition follows from Theorem~3.2 of~\cite{AGG}, 
which shows that the tropical polyhedral cone $\sK_\lambda$
contains a vector $x$ such that $x_j\neq \mpzero$ if, 
and only if, $\chi_j(g_\lambda) \geq 0$, 
i.e. column node $j$ is a winning initial state for Player Max 
in the mean payoff game associated with 
the system of inequalities $A x\leq B x$, $\lambda d x \leq c x$. 
\end{proof}

The next result considers the case of vectors with only finite entries. 

\begin{proposition}\label{prop-impl-finite}
The implication $A x\leq B x \implies c x \leq d x$ 
does not hold for all $x\in \R^n$ if, and only if, 
$\min_{j\in[n]} \chi_j(g_\lambda) \geq 0$ for some $\lambda > 0$.
\end{proposition}

\begin{proof}
The implication $A x\leq B x \implies c x \leq d x$ 
does not hold for all $x\in \R^n$ if, and only if, there exists 
a scalar $\lambda>0$ such that the tropical polyhedral cone $\sK_\lambda$ 
defined in~\eqref{Klambda} contains a finite vector.
Then, as in the proof of Proposition~\ref{prop-rhoS}, 
the result follows from Theorem~3.2 of~\cite{AGG} because 
this theorem shows that $\sK_\lambda$ contains a finite vector if, 
and only if, $\chi_j(g_\lambda) \geq 0$ for all $j\in [n]$.  
\end{proof}

The situation in which $cx=dx=\mpzero $ for some non-identically $\mpzero $ 
vector $x$ in the tropical cone 
$\left\{ x\in \maxplus^n\mid Ax\leq Bx \right\}$ 
appears to be degenerate. Hence, 
{\em in the sequel we shall use the following technical assumption}, 
which, as we shall shortly see, implies no loss of generality.

\begin{assumption}\label{infty_assump}
If $x \in \maxplus^n$ is such that $A x \leq B x$ and $c x=d x=\mpzero $, 
then $x$ is the identically $\mpzero $ vector.
\end{assumption}

The following lemma shows that
the vector $d$ may always be required to be finite.
Then, the previous assumption is trivially satisfied.

\begin{lemma}
Let the constant $M$ be defined as in Proposition~\ref{prop-genex}, 
and define the vector $d'\in\R^n$ by
\[
d'_i\defi \begin{cases}d_i& \text{if }d_i\neq \mpzero \\
-M-1+\min_{j\in[n],c_j\neq \mpzero } c_j &\text{otherwise.}
\end{cases}
\]
Then $d$ can be replaced by $d'$ without changing
the validity of the implication $Ax\leq Bx\implies cx\leq dx$.
\end{lemma}

\begin{proof}
Since $d\leq d'$, the implication in which $d'$ appears
is weaker than the one with $d$. Assume that the latter implication
does not hold. Then, by Corollary~\ref{coro-smallcex},
there is a vector $y$ such that $A y \leq B y$ and $d y < c y$, 
and this vector satisfies~\eqref{e-cbound}.
Let $j$ denote any index such that $c_j y_j=c y$. 
Using~\eqref{e-cbound} we deduce that 
$d'_k y_k \leq (-M-1) c_j y_k < c_j y_j = c y$ 
for any $k$ such that $d_k=\mpzero $, and so $d' y < c y$,
showing that the implication in which $d$ is replaced by $d'$ does
not hold.
\end{proof}

Thanks to Assumption~\ref{infty_assump}, 
the validity of the implication $Ax\leq Bx\implies cx\leq dx$ 
can now be characterized in terms of the spectral radius. 

\begin{proposition}\label{prop-rho}
The system of inequalities $A x\leq B x$ does not imply 
the scalar inequality $c x \leq d x$ if, 
and only if, $\rho(g_\lambda) \geq 0$ for some $\lambda > 0$.
\end{proposition}

\begin{proof}
Since Assumption~\ref{infty_assump} holds, 
there is a vector $y$ such that $A y\leq B y$ and $c y>d y$ if, 
and only if, there exists a scalar $\lambda>0$ such that 
the tropical polyhedral cone $\sK_\lambda$ defined in~\eqref{Klambda} 
is not trivial (i.e., not reduced to the identically $\mpzero$ vector).
Then, the conclusion follows from Theorem~3.1 of~\cite{AGG}, 
which shows that $\sK_\lambda$ is not trivial if, and only if, 
the mean payoff game associated with the system of inequalities 
$A x\leq B x$, $\lambda d x \leq c x$ has at least one 
winning initial state for Player Max, 
which by Lemma~2.8 and Proposition~2.11 of the same paper holds if, 
and only if, the associated dynamic programming operator $g_\lambda$ 
has spectral radius at least $0$.
\end{proof}

We call the map $\lambda \mapsto \rho(g_\lambda)$ the \NEW{spectral function}.
The idea of considering a parametric spectral radius somehow
similar to this one appears in~\cite{sergeev},
where it is used to solve a different problem (two-sided eigenproblem).
We next indicate some elementary properties of the spectral function.

\begin{lemma}
The spectral function $\lambda\mapsto \rho(g_\lambda)$ is non-increasing.
\end{lemma}

\begin{proof}
We claim that for all $k\in \N $ and $x \in \R^n$, 
the map $\lambda \mapsto g^k_\lambda(x)$ is order reversing from $\R$ to $\R^n$, 
meaning that $\lambda \leq \mu$ implies $g^k_\mu(x)\leq g^k_\lambda(x)$.
We prove this claim by induction. For $k=1$, this property is immediate.
Assume that our claim holds for $k=r$. Then, if $\lambda\leq \mu$, 
we have $g_\mu^{r+1}(x)=g_\mu(g^r_\mu(x))\leq g_\lambda(g^r_\mu(x))$,
and since $y\mapsto g_\lambda(y)$ is order preserving, 
using the induction hypothesis,
 we conclude that $g_\lambda(g^r_{\mu}(x))\leq g_\lambda^{r+1}(x)$, 
proving our claim. Thus, 
it follows from~\eqref{e-weakcw} that the map $\lambda\mapsto \rho(g_\lambda)$, 
which is a pointwise limit of non-increasing functions, is non-increasing.
\end{proof}

Now we show that the spectral function is piecewise affine, by
describing explicitly a complete family of ``tangent'' affine maps.
This description involves the notion of strategy.

We call \NEW{strategy} for Player Min a map $\sigma$ which assigns
to each column node $j\in [n]$ a row node $\sigma(j)\in [p+1]$ such that 
$A_{\sigma(j)j}\in \R$ if $\sigma(j)\in [p]$, or $d_j\in \R$ if $\sigma(j)=p+1$.
We associate with the strategy $\sigma$ the map $g_\lambda^\sigma$ defined by:  
\[ 
[g_\lambda^\sigma(x)]_j \defi 
\begin{cases}
-A_{\sigma(j)j}+\max_{k\in[n]}(B_{\sigma(j)k}+x_k)& \text{ if } \sigma(j)\in [p]\enspace , \\
-\lambda -d_j+\max_{k\in[n]}(c_k+x_k)& \text{ if } \sigma(j)= p+1 \enspace .
\end{cases}
\]
Observe that with the tropical notation, 
\[ 
[g_\lambda^\sigma(x)]_j = 
\begin{cases}
A_{\sigma(j)j}^{-1}B_{\sigma(j)\cdot } \, x & \text{ if }\sigma(j)\in [p]\enspace , \\ 
\lambda^{-1}d_j^{-1} c x & \text{ if }\sigma(j)= p+1 \enspace , 
\end{cases}
\]
so $g_\lambda^\sigma$ is a tropical linear map.  
By the definition of $g_\lambda$, 
we have the following selection property,
which holds for all $\lambda\in \R$,
\begin{align}\label{e-selection}
\makebox{For each } x\in \maxplus^n\makebox{ there exists a strategy }\sigma 
\makebox{ such that } g_\lambda(x)=g_\lambda^\sigma(x) \enspace .
\end{align}
We shall use the following result, which may be thought
of as a variant
of the ``duality theorem'' established
in~\cite{gg98a} (see also~\cite{gg0}).

\begin{lemma}\label{lem-morphism}
Let $g:\maxplus^n \to \maxplus^n$ be a continuous order preserving map 
that commutes with the addition of a scalar. 
Assume that $g$ is the pointwise infimum of a family of maps 
$\{ h^\sigma \}_{\sigma\in\Sigma}$ all of which are continuous,  
order preserving and commute with the addition of a scalar.  
If for each $x\in \maxplus^n$ there exists $\sigma \in \Sigma $ 
such that $g(x)=h^\sigma(x)$ (selection property), then
\begin{align}\label{e-morphism}
\rho(g)=\min_{\sigma\in\Sigma} \rho(h^\sigma) \enspace .
\end{align}
\end{lemma}

\begin{proof}
It follows from the characterization~\eqref{e-weakcw} that the map
$g\mapsto \rho(g)$ is non-decreasing. Hence, 
$\rho(g)\leq \min_{\sigma\in\Sigma} \rho(h^\sigma)$. 

The Collatz-Wielandt formula (see~\cite[Lemma~2.8]{AGG}) shows that 
the spectral radius of a map $f:\maxplus^n \to \maxplus^n$ which is continuous, 
order preserving and commutes with the addition of a scalar 
satisfies the equality: 
\begin{align}\label{CWformula}
\rho(f)=\inf \left\{\mu \in \R \mid \exists y\in \R^n,f(y)\leq \mu + y \right\} \enspace .  
\end{align}
Therefore, for any $\alpha > 0$, there exists a vector $y \in \R^n$  
such that $g(y)\leq \rho(g)+\alpha +y$. Using the selection property, we deduce that 
$h^\sigma(y)=g(y)\leq \rho(g)+\alpha +y$ for some $\sigma \in \Sigma$. 
Hence, by~\eqref{CWformula} we have $\rho(h^\sigma)\leq \rho(g)+\alpha$. 
Since this holds for any $\alpha >0$, 
we conclude that $\min_{\sigma\in\Sigma} \rho(h^\sigma)\leq \rho(g)$.
\end{proof}

By the previous lemma, it follows that 
\begin{align}\label{e-morphism2}
\rho(g_\lambda)=\min_{\sigma}\rho(g_\lambda^\sigma)
\end{align}
for all $\lambda\in \R$, 
where the minimum is taken over the set of all strategies for Player Min. 

A strategy $\sigma$ for Player Min defines a ``one player sub-game'' 
in which only Player Max has to make choices. This sub-game corresponds to the
sub-graph $\sG_\lambda^\sigma$ of $\sG_\lambda$ in which for each column node $j$ 
we delete all the arcs leaving this node 
except the one going to row node $\sigma(j)$. 
Define the length of a circuit in the digraph $\sG_\lambda^\sigma$ 
to be the number of column nodes that it contains. 
Then, it follows from the max-plus spectral theorem
(see for example~\cite{cuni79,BCO92}) that $\rho(g_\lambda^\sigma)$ coincides
with the maximal weight-to-length ratio of circuits in the digraph
$\sG_\lambda^\sigma$, see also~\cite{gg} for a discussion adapted to the present
setting.

The classical Farkas lemma gives a simple ``certificate'' that a linear 
inequality over the reals is implied by a finite family of linear inequalities. 
This certificate consists of nonnegative coefficients 
(Lagrange multipliers) expressing the 
given inequality as a linear combination of the inequalities in the family. 
The following result does the same in the tropical setting. However, 
the certificate is now of a different nature: the collection of Lagrange 
multipliers is replaced by a strategy. 
 
\begin{theorem}[Tropical analogue of Farkas lemma]\label{theo-farkas}
The implication $Ax\leq Bx\implies cx\leq dx$ holds if, and only if,
there exists a strategy $\sigma$ for Player Min such that every circuit
in the digraph $\sG_0^\sigma$ has nonpositive weight, and if a circuit
in this digraph has zero weight, then it passes through row node $p+1$. 
\end{theorem}

\begin{proof}
Observe that for each strategy $\sigma$,
the map $\lambda\mapsto \rho(g_\lambda^\sigma)$ is piecewise affine.
Actually, it is given by the maximal weight-to-length ratio of the 
(elementary) circuits in the digraph $\sG_\lambda^\sigma$, 
and the weight of each of these circuits is an affine function of $\lambda$. 
Since there is a finite number of strategies, 
we deduce from~\eqref{e-morphism2} that there exist a strategy $\sigma $ 
and a positive number $\tilde{\lambda}$ such that
\[
\rho(g_\lambda)=\rho(g_\lambda^\sigma)\enspace  ,
\]
for all $\lambda\in [0,\tilde{\lambda}]$.

As the spectral function is non-increasing,
by Proposition~\ref{prop-rho}
it follows that the implication
$Ax\leq Bx\implies cx\leq dx$ holds if, and only if, 
$\rho(g_\lambda)=\rho(g_\lambda^\sigma)<0$ for all $\lambda \in (0,\tilde{\lambda}]$. 
Now, using the characterization of the spectral radius of the
tropical linear map $g_\lambda^\sigma$ as the maximal weight-to-length
ratio of circuits in the digraph $\sG_\lambda^\sigma$, 
from $\rho(g_\lambda^\sigma)<0$ for all $\lambda \in (0,\tilde{\lambda}]$,
we deduce that every circuit in $\sG_0^\sigma$ must have nonpositive weight.
Otherwise, by continuity of $\lambda \mapsto \rho(g_\lambda^\sigma)$, 
we would have $\rho(g_\lambda^\sigma)>0$ for some $\lambda >0$, which is nonsense. We also deduce that every circuit of zero weight in $\sG_0^\sigma$ (if any) must contain an arc on which the parameter $-\lambda$ appears, i.e., an arc of weight $-d_j$
from some column node $j$ to row node $p+1$.
Otherwise, by definition of $\sG_\lambda^\sigma$, 
the weight of this circuit would also be zero in $\sG_\lambda^\sigma$ for 
every $\lambda >0$, contradicting the fact that $\rho(g^\sigma_\lambda)<0$.
This shows that the condition of the theorem is necessary. 

Conversely, assume that there exists a strategy $\sigma $ 
satisfying the condition of the theorem. Then, 
by the characterization of the spectral radius $\rho(g_\lambda^\sigma)$ 
as the maximal weight-to-length ratio of circuits in $\sG_\lambda^\sigma$, 
it follows that $\rho(g_\lambda^\sigma)<0$ for all $\lambda >0$. 
Since by~\eqref{e-morphism2} we have 
$\rho(g_\lambda)\leq \rho(g_\lambda^\sigma)<0$ for all $\lambda >0$, 
from Proposition~\ref{prop-rho} we conclude that the implication
$Ax\leq Bx\implies cx\leq dx$ holds.   
\end{proof} 

We now state a dual result, in which strategies are used
to certify that the implication does not hold.
We shall consider a strategy $\pi$ for Player Max, which
is a map from the set of row nodes to the set of column nodes, 
assigning to each row node $i$ a unique arc leaving it,
with destination to some column node $\pi(i)$. Each such strategy
defines a new sub-game, by erasing all arcs leaving row node $i$ but the one
going to column node $\pi(i)$. We denote by $\maxpolgraph{\pi}{\sG_\lambda}$ the corresponding
sub-graph of $\sG_\lambda$. Define now the map $\maxpolmap{\pi}{g_\lambda}$ by
\[
[\maxpolmap{\pi}{g_\lambda}(x)]_j \defi
\min \Big( \min_{i\in[p],A_{i j}\in \R}\big( -A_{i j}+B_{i\pi(i)}+x_{\pi(i)} \big),
-\lambda -d_j+c_{\pi(p+1)}+x_{\pi(p+1)} \Big) 
\]
if $d_j \in \R $,  
\[
[\maxpolmap{\pi}{g_\lambda}(x)]_j \defi
\min_{i\in[p],A_{i j}\in \R}\big( -A_{i j}+B_{i\pi(i)}+x_{\pi(i)} \big) 
\]
otherwise.
Observe that for every strategy $\pi$ for Player Max, $\maxpolmap{\pi}{g_\lambda}$ 
is a self-map of $\R^n$  that commutes with the addition of a scalar,
and it has a unique (continuous) order preserving extension to
$\maxplus^n$, which is also denoted by $\maxpolmap{\pi}{g_\lambda}$.
Hence, the definition of the additive spectral radius, $\rho$, applies to the map
$\maxpolmap{\pi}{g_\lambda}$.
Since $\maxpolmap{\pi}{g_\lambda} \leq g_\lambda$ and for each $x\in \maxplus^n$  
there exists a strategy $\pi $ such that $g_\lambda (x)=\maxpolmap{\pi}{g_\lambda} (x)$, 
we deduce that 
\begin{align}\label{e-morphism-dual}
\rho(g_\lambda)=\max_{\pi}\rho (\maxpolmap{\pi}{g_\lambda}) 
\end{align}
for all $\lambda\in \R$. Indeed, we have already noted
that the characterization~\eqref{e-weakcw} implies
that the spectral radius $\rho$ of a map is an order preserving
function of this map, and so
\begin{align}\label{e-morphism-dual2}
\rho(g_\lambda)\geq \max_{\pi}\rho (\maxpolmap{\pi}{g_\lambda}) \enspace . 
\end{align}
To see that the equality is attained, the argument is dual
to the proof of Lemma~\ref{lem-morphism} above. If $u$ is an eigenvector
of $g_\lambda$, so that $g_\lambda(u)=\rho(g_\lambda)+u$, 
using the former selection property, we deduce that
$\maxpolmap{\pi}{g_\lambda}(u)=g_\lambda(u)=\rho(g_\lambda)+u$ for some strategy $\pi$
for Player Max, and so, $\rho(g_\lambda)\leq \rho(\maxpolmap{\pi}{g_\lambda})$,
which implies that the equality holds in~\eqref{e-morphism-dual2}.

By applying~\eqref{e-rho-isamax} to $g_\lambda^\pi$, we get 
\begin{align}
\rho (\maxpolmap{\pi}{g_\lambda})=\max_{j\in [n]} \chi_j(\maxpolmap{\pi}{g_\lambda}) \enspace.
\label{e-carac-chi}
\end{align}
Then, using~\eqref{e-morphism-dual} 
and~\eqref{e-carac-chi}, we obtain the following immediate 
consequence of Proposition~\ref{prop-rho}.  
 
\begin{proposition}\label{prop-farkasnew}
The implication $A x\leq B x\implies c x\leq d x$ does not hold if, 
and only if, there exists a strategy $\pi$ for Player Max, 
a column node $j\in [n]$ and a scalar $\lambda>0$ such that 
$\chi_j(\maxpolmap{\pi}{g_\lambda})\geq 0$.  \qed
\end{proposition}

The cycle time $\chi(g_\lambda^\pi)$ has a simple characterization.
For each strongly connected component $C$ of the digraph $\sG_\lambda^\pi$,
let $\nu_C$ denote the minimal weight-to-length ratio of the circuits in $C$
(the length being defined as the number of column nodes in the circuit).
Then, it is known that
\begin{align}
\chi_j(g_\lambda^\pi)=\min_{C} \nu_C\label{eq-formulachi}
\end{align}
where the minimum is taken over all the strongly connected components $C$
to which there is a path in $\sG_\lambda^\pi$ from column node $j$ (see for instance~\cite[\S~1.4]{gg} or~\cite{coc-98}). Recall that 
every minimal ratio $\nu_C$ can be computed in polynomial time by Karp's algorithm.

Arguing as in the proof of Theorem~\ref{theo-farkas}
and using~\eqref{eq-formulachi}, we arrive at the following result,
which expresses Proposition~\ref{prop-farkasnew} in combinatorial terms.
This is somehow dual to Theorem~\ref{theo-farkas}. 

\begin{corollary}\label{prop-farkas}
The implication $A x\leq B x\implies c x\leq d x$ does not hold if, 
and only if, 
there exists a strategy $\pi$ for Player Max with the following property:
in the digraph $\maxpolgraph{\pi}{\sG_0}$ 
there exists a column node $j\in [n]$ such that 
every circuit reachable from $j$ has nonnegative weight, 
and if a circuit of zero weight is reachable from $j$, 
then it does not pass through row node $p+1$.\qed
\end{corollary}

\begin{example}\label{ex-deduce}

Consider the inequalities $x_1\oplus (-2)x_3 \leq x_2$
and $x_2 \leq (-3)x_1\oplus x_3$. 
We next apply the previous method to show that these inequalities 
imply the inequality $x_1\oplus x_2 \leq x_3$. 
The tropical cones associated with these inequalities are illustrated 
in Figure~\ref{FigDeduc} in barycentric coordinates. 

\begin{figure}
\renewcommand{\baryx}{x_1}
\renewcommand{\baryy}{x_2}
\renewcommand{\baryz}{x_3}
\begin{center}
\begin{minipage}[b]{0.45\textwidth}
\begin{tikzpicture}%
[scale=0.65,>=triangle 45
,vtx/.style={mygreen},
ray/.style={myred}]
\equilateral{7}{100};

\barycenter{g1}{\expo{0}}{\expo{0}}{\expo{0}};
\barycenter{g4}{\expo{0}}{0}{\expo{2}};
\barycenter{g5}{\expo{0}}{\expo{0}}{\expo{2}};
\barycenter{g7}{\expo{3}}{\expo{0}}{0};
\barycenter{g8}{\expo{3}}{\expo{0}}{\expo{0}};
\barycenter{g9}{\expo{0}}{0}{\expo{0}};
\barycenter{g10}{0}{\expo{0}}{\expo{2}};

\barycenter{e1}{\expo{0}}{0}{0};
\barycenter{e2}{0}{\expo{0}}{0};
\barycenter{e3}{0}{0}{\expo{0}};

\barycenter{f1}{0}{\expo{0}}{\expo{0}};
\barycenter{f3}{\expo{0}}{\expo{0}}{0};

\filldraw[lightgray,draw=black,opacity=0.2,thin] (g5) -- (g10) -- (e2) -- (f3) -- cycle;

\node at (8,3.5) {
$
\begin{aligned}
x_1\oplus (-2) x_3 & \leq x_2
\end{aligned}
$};
\end{tikzpicture}
\end{minipage}
\begin{minipage}[b]{0.45\textwidth}
\begin{tikzpicture}%
[scale=0.65,>=triangle 45
,vtx/.style={mygreen},
ray/.style={myred}]
\equilateral{7}{100};

\barycenter{g1}{\expo{0}}{\expo{0}}{\expo{0}};
\barycenter{g4}{\expo{0}}{0}{\expo{2}};
\barycenter{g5}{\expo{0}}{\expo{0}}{\expo{2}};
\barycenter{g7}{\expo{3}}{\expo{0}}{0};
\barycenter{g8}{\expo{3}}{\expo{0}}{\expo{0}};
\barycenter{g9}{\expo{0}}{0}{\expo{0}};

\barycenter{e1}{\expo{0}}{0}{0};
\barycenter{e2}{0}{\expo{0}}{0};
\barycenter{e3}{0}{0}{\expo{0}};

\barycenter{f1}{0}{\expo{0}}{\expo{0}};
\barycenter{f3}{\expo{0}}{\expo{0}}{0};

\filldraw[lightgray,draw=black,opacity=0.2,thin] (g7) -- (e1) -- (e3) -- (f1) -- (g8) -- cycle;

\node at (8,3.5) {
$
\begin{aligned}
x_2 & \leq (-3)x_1\oplus x_3
\end{aligned}
$};
\end{tikzpicture}
\end{minipage}
\end{center}
\begin{center}
\begin{minipage}[b]{0.45\textwidth}
\begin{tikzpicture}%
[scale=0.65,>=triangle 45
,vtx/.style={mygreen},
ray/.style={myred}]
\equilateral{7}{100};

\barycenter{g1}{\expo{0}}{\expo{0}}{\expo{0}};
\barycenter{g4}{\expo{0}}{0}{\expo{2}};
\barycenter{g5}{\expo{0}}{\expo{0}}{\expo{2}};
\barycenter{g7}{\expo{3}}{\expo{0}}{0};
\barycenter{g8}{\expo{3}}{\expo{0}}{\expo{0}};
\barycenter{g9}{\expo{0}}{0}{\expo{0}};

\barycenter{e1}{\expo{0}}{0}{0};
\barycenter{e2}{0}{\expo{0}}{0};
\barycenter{e3}{0}{0}{\expo{0}};

\barycenter{f1}{0}{\expo{0}}{\expo{0}};
\barycenter{f3}{\expo{0}}{\expo{0}}{0};

\filldraw[gray,draw=black,opacity=0.6,very thick] (g1) -- (f1) -- (g10) -- (g5) -- cycle;

\node at (8,3.5) {
$
\begin{aligned}
&x_1\oplus(-2)x_3  \leq x_2\\
&x_2  \leq (-3)x_1\oplus x_3
\end{aligned}
$};
\end{tikzpicture}
\end{minipage}
\begin{minipage}[b]{0.45\textwidth}
\begin{tikzpicture}%
[scale=0.65,>=triangle 45
,vtx/.style={mygreen},
ray/.style={myred}]
\equilateral{7}{100};

\barycenter{g1}{\expo{0}}{\expo{0}}{\expo{0}};
\barycenter{g4}{\expo{0}}{0}{\expo{2}};
\barycenter{g5}{\expo{0}}{\expo{0}}{\expo{2}};
\barycenter{g7}{\expo{3}}{\expo{0}}{0};
\barycenter{g8}{\expo{3}}{\expo{0}}{\expo{0}};
\barycenter{g9}{\expo{0}}{0}{\expo{0}};

\barycenter{e1}{\expo{0}}{0}{0};
\barycenter{e2}{0}{\expo{0}}{0};
\barycenter{e3}{0}{0}{\expo{0}};

\barycenter{f1}{0}{\expo{0}}{\expo{0}};
\barycenter{f3}{\expo{0}}{\expo{0}}{0};

\filldraw[gray,draw=black,opacity=0.6,very thick] (g1) -- (f1) -- (g10) -- (g5) -- cycle;

\filldraw[gray,draw=black,opacity=0.4,thin] (g1) -- (g9) -- (e3) -- (f1) -- cycle;

\node at (8,3.5) {
$
\begin{aligned}
x_1 \oplus x_2 \leq x_3
\end{aligned}
$};
\end{tikzpicture}
\end{minipage}
\end{center}

\caption{The final tropical linear inequality follows from the first two ones, although it cannot be obtained from them by tropical linear combinations.}
\label{FigDeduc}
\end{figure}
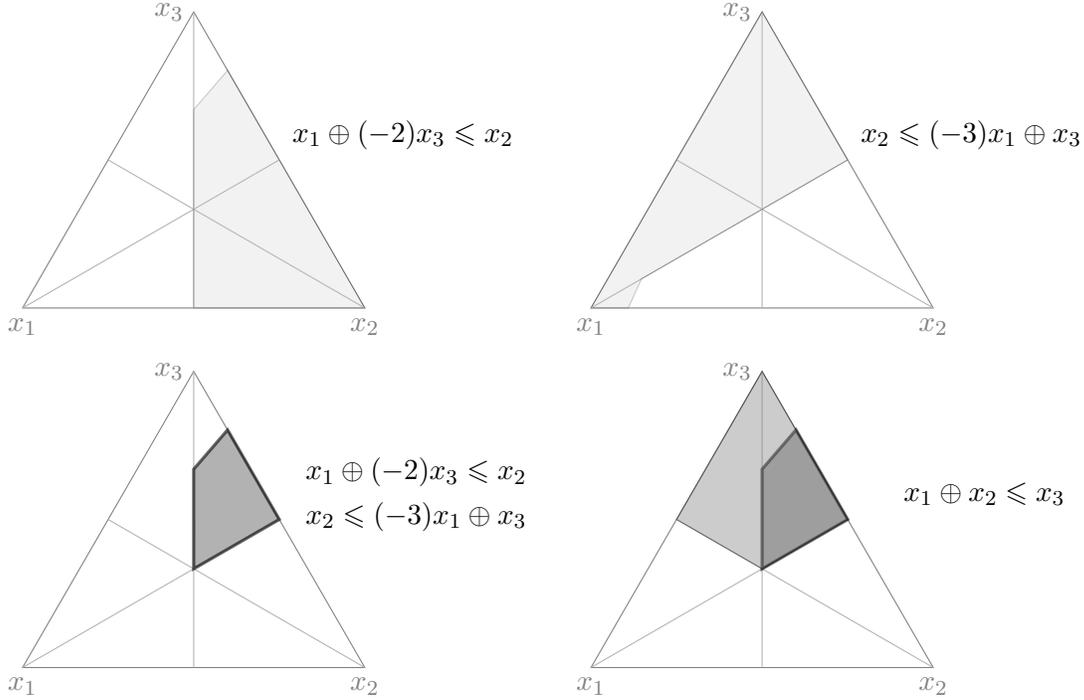

Observe that in this case, we have 
\[
A=\begin{pmatrix}
0 & \mpzero & -2 \\ 
\mpzero & 0 & \mpzero 
\end{pmatrix} \enspace , \enspace 
B=\begin{pmatrix}
\mpzero & 0 & \mpzero  \\ 
-3 & \mpzero & 0 
\end{pmatrix} \enspace , \enspace 
c=\begin{pmatrix}
0 & 0 & \mpzero 
\end{pmatrix} \enspace , \enspace 
\makebox{ and } \enspace 
d=\begin{pmatrix}
\mpzero & \mpzero & 0 
\end{pmatrix} \enspace . 
\]
The associated bipartite digraph $\sG_\lambda$ is depicted in 
Figure~\ref{BipartiteDigraph}, 
where row nodes are represented by squares and column nodes by circles. 
If we consider the strategy $\sigma $ for Player Min defined by 
$\sigma (1)=1$, $\sigma (2)=2$ and $\sigma (3)=3$, 
it can be checked that all circuits in $\sG_0^\sigma $ have nonpositive weight 
and that any circuit of zero weight passes through row node $p+1=3$. 
The latter can also be checked by deleting row node $p+1=3$ from $\sG_0^\sigma $, and the arcs adjacent to it (dotted arcs on Figure~\ref{BipartiteDigraph}, middle) because the resulting digraph contains only one circuit and this circuit has negative weight. 
Therefore, by Theorem~\ref{theo-farkas} we conclude that 
$A x \leq B x$ implies $c x\leq d x$, as can be seen in Figure~\ref{FigDeduc}. 

Consider now the inequality $1x_1\oplus x_2 \leq x_3$ instead  of $x_1\oplus x_2 \leq x_3$, 
so that in Figure~\ref{BipartiteDigraph} the weight of the arc connecting row node  $p+1=3$ 
with column node $1$ is now $1$ instead of $0$. 
If we define the strategy $\pi$ for Player Max by $\pi(1)=2$, $\pi(2)=3$ and $\pi(3)=1$, 
then all circuits in $\maxpolgraph{\pi}{\sG_0} $ have positive weight. Thus, 
from Corollary~\ref{prop-farkas}, it follows that 
$A x \leq B x$ does not imply $c x\leq d x$. 
Observe that $x=(0,0,0)^t$ satisfies  
$A x \leq B x$ but not $c x\leq d x$. 
 
\begin{figure}
\begin{center}
\input{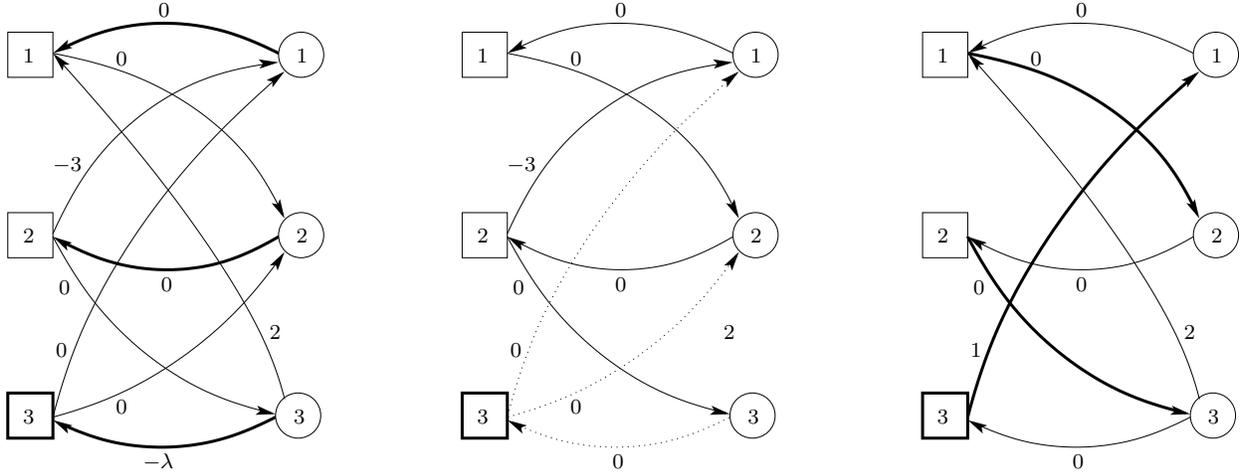}
\end{center}\caption{Illustration of Theorem~\ref{theo-farkas} and Corollary~\ref{prop-farkas}. (Left): The parametric game $\sG_\lambda $ allowing one to check the implication shown in Figure~\ref{FigDeduc}; row nodes are represented by squares (Max plays), column nodes by circles (Min plays); the strategy $\sigma$ for Player Min certifying the implication is shown in bold. (Middle): The sub-game $\sG_0^\sigma$ induced by this strategy: there are no circuits of positive weight, and every circuit of zero weight passes through the special row node $p+1$. (Right): The sub-game $\maxpolgraph{\pi}{\sG_0}$ induced by a strategy $\pi$ for Player Max (in bold), certifying that for a perturbed vector $c$ (the weight $c_1$ of the arc connecting row node $3$ with column node $1$ is now $1$), the implication no longer holds. Every circuit has now positive weight.}
\label{BipartiteDigraph} 
\end{figure}

\end{example}

\begin{remark}
Let us now restrict our attention to instances in which 
the entries of the matrices $A,B$ and the vectors
$c,d$ belong to $\Q_{\max }$. 
Then, Theorem~\ref{theo-farkas} implies that the problem
 ``does $Ax\leq Bx\implies cx\leq dx$ hold?''
is in {\sc NP}. Indeed, any strategy $\sigma$ for Player Min
satisfying the condition of the theorem provides a certificate
which can be checked in polynomial time. 
To see this, it suffices to compute the maximal weight-to-length 
ratio of circuits in $\sG_0^\sigma$, which can be done by applying
Karp's algorithm to every strongly connected component of $\sG_0^\sigma$.
To be valid, the certificate requires these maximal weight-to-length ratios
to be nonpositive. Moreover, if one of these maximal weight-to-length ratios
is zero (indeed, to be valid, only the one corresponding to the 
strongly connected component containing row node $p+1$ could be zero), 
we must check whether there is a circuit of zero 
weight in $\sG_0^\sigma$ which does not pass through row node $p+1$.
This can be verified by deleting row node $p+1$ and the arcs connected to it from the digraph 
$\sG_0^\sigma$ and computing the maximal weight-to-length ratio of circuits in 
the resulting digraph. 
\end{remark}

\begin{remark}
Similarly, Corollary~\ref{prop-farkas} implies that the problem
 ``does $Ax\leq Bx\implies cx\leq dx$ hold?'' belongs to {\sc co-NP}. 
A negative certificate (disqualification)
is now a strategy $\pi$ for Player Max and the validity
of such a certificate can still be checked in polynomial time. 
We next only sketch the argument (which is more involved than in the preceding case), leaving details to the reader.
First, apply~\eqref{eq-formulachi} to compute $\chi(g_0^\pi)$.
This requires calling Karp's algorithm at most $n$ times, and 
can therefore be done in polynomial time. If $\chi_j(g_0^\pi)> 0$ 
for some $j\in [n]$,
the certificate is valid. If $\chi_j(g_0^\pi)< 0$  
for all $j\in [n]$, 
the certificate is invalid. 
If none of the previous conditions is satisfied, 
for each $j\in [n]$ such that $\chi_j(g_0^\pi)=0$
we proceed as follows. 
Assume that the strongly connected component $C$ containing row node $p+1$ 
is reachable from $j$ and that $\nu_C=0$ (otherwise, the certificate is valid).   
Now, consider a potential transformation, which consists, for every arc $(i,j)$ in $C$, in replacing the weight $w_{ij}$ of this arc by $w'_{ij}=u_i+w_{ij}-u_j$, where a real number $u_k$ (the potential) must be chosen for each row or column node $k$. Obviously, this transformation does not change
the weight of circuits, and a fortiori, the weight-to-length ratios of circuits.
It follows from
the duality theorem in linear programming that we can find a potential
such that $w_{ij}'\geq 0$ for all arcs $(i,j)$ in $C$, and then the circuits
of zero weight in $C$ are precisely the circuits composed of those arcs $(i,j)$ such that $w'_{ij}=0$. Now, we delete all the arcs but these ones,
which yields a sub-graph. If no circuit in this sub-graph passes through row node $p+1$, the certificate is valid.
\end{remark}
\begin{remark}
In the proof of the tropical Farkas lemma, following the route of~\cite{AGG},
we used techniques of non-linear Perron-Frobenius theory showing that the spectral radius $\rho$ is a ``morphism''
with respect to the infimum or supremum of families of maps having a selection property, meaning that~\eqref{e-morphism2} and~\eqref{e-morphism-dual} hold.
One might look for an alternative and more combinatorial proof. 
Indeed, we may define directly the value of the mean payoff
game for every initial state as in~\cite{liggettlippman} and then define $\rho$ as
the maximum of this value over all the initial states, instead as the limit
of the value per time unit of the finite horizon game~\eqref{e-weakcw}.
As pointed out above, the theorem of Kohlberg~\cite{kohlberg} implies that the two definitions of $\rho$ coincide, in other words, that the value commutes with the limit. Instead of Kohlberg's theorem,
one might exploit the combinatorial approach of Gurvich, Karzanov, and Khachiyan~\cite{gurvich}, which relies on potentials solving certain systems of inequalities. It should be noted that the potential vector returned by their algorithm is of the same nature as an invariant half-line (the basepoint of an half-line determines a potential).  We finally point out that M\"ohring, Skutella, and Stork~\cite{mohring} established an equivalence between mean payoff games
and certain scheduling problems with and/or constraints,
which turn out to be equivalent to the existence
of finite vectors in a tropical polyhedron. Some techniques used
in~\cite{mohring} might also yield alternative approaches to the present problems.
\end{remark}

When the entries of the matrices $A,B$ and the vectors $c,d$ belong
to $\Z_{\max }$, there turns out to be a simpler characterization. 

\begin{proposition}\label{prop-integer}
The implication $Ax\leq Bx\implies cx\leq dx$ does not hold
if, and only if, $\rho(g_1)\geq 0$.
\end{proposition}

\begin{proof}
By Proposition~\ref{prop-integerisenough}, if the implication does not hold, 
there is a vector $y\in \Z_{\max}^n$ such that $A y\leq B y$ and $d y<c y$.
Since the finite entries of $d$, $c$ and $y$ are integers, we
must have $\lambda d y\leq c y$ for $\lambda =1$. 
It follows that $\rho(g_1)\geq 0$. 

Conversely, if $\rho(g_1)\geq 0$, 
the system $Ax\leq Bx, \lambda dx\leq cx$ for $\lambda =1$ 
has at least one non-identically $\mpzero $ solution $y\in \maxplus^n$. 
Then, by Assumption~\ref{infty_assump},
we must have $c y>\mpzero $. It follows that $d y<c y$, 
showing that the implication does not hold. 
\end{proof}

Recall that for a given $\lambda$, $\rho(g_\lambda)$, which
is the value of a mean payoff game, can be computed in pseudo-polynomial 
time by standard value iteration arguments. See~\cite{zwick}
and also~\cite[Section~3.2]{AGG} for a refinement using the Collatz-Wielandt property. The existence of a polynomial time algorithm is an open question.

By combining the results of~\cite{AGG,aggut10} and Proposition~\ref{prop-integer}, we arrive at the following result, in which the matrices and vectors
are still required to have entries in $\Z_{\max }$, and payments
of games are still integers.

\begin{corollary}\label{coro-equiv}
The problem ``does a mean payoff game have at least one winning initial 
state?'' (i.e., if $g$ is the dynamic programming
operator of a mean payoff game, does $\rho(g)\geq 0$ hold?) is polynomial time
equivalent to the problem of deciding whether the implication
\[
Ax\leq Bx \implies cx\leq dx
\]
holds.
\end{corollary}

\begin{proof}
Theorem~3.1 of~\cite{AGG} shows that checking whether $\rho(g)\geq 0$ is 
polynomial time equivalent to deciding whether an associated tropical
polyhedral cone $\{x\in \maxplus^n\mid Ax\leq Bx\}$ is not reduced to the trivial
(identically $\mpzero $) vector. The latter reduces to checking whether
the implication $A x\leq B x\implies c x\leq d x$ does not hold, where
$c,d$ are any pair of vectors such that $c_j > d_j > \mpzero$ for all $j \in [n]$.

Conversely, assume we have an oracle allowing us to decide whether $\rho(g)\geq 0$ for any dynamic programming operator $g$ of a mean payoff game with integer rewards in which $n$ states belong to Player Min. By Proposition~\ref{prop-integer}, it suffices to apply this oracle to the map $g_{1}$ to decide whether the implication holds.
\end{proof}

\begin{remark}
When $\rho(g_0)<0$, the unique solution of the system $A x\leq B x$, 
$d x\leq c x$ is the identically $\mpzero $ vector, and vice versa. 
Thus, the stronger implication
\[
A x\leq B x ,\; x\not\equiv \mpzero  \implies c x < d x 
\]
is characterized by $\rho(g_0)<0$.
\end{remark}

Proposition~\ref{prop-integer} leads to a greedy algorithm
to construct non-redundant
systems of inequalities defining a tropical polyhedral cone $\sK$.
With this aim, we apply the following procedure.

{\em(i)}\/ We start from the extreme rays of the polar of $\sK$,
which correspond to a finite family of inequalities $a^jx\leq b^jx$, $j\in J$.

{\em(ii)}\/ We check, using the characterization of the proposition
(by computing $\rho(g_{1})$, the value of a mean payoff game), 
whether any of these inequalities is implied by the other ones.
If this is the case, we delete the inequality from the list.

In this way, we arrive at a minimal set of inequalities
defining $\sK$. The next example shows that such a set
is not unique: running the previous greedy algorithm by
scanning the inequalities in different orders
yields incomparable minimal sets of defining inequalities.

For instance, if $\sK:=\mathcal{P}(5,4)$ is the tropical cyclic
polyhedral cone with five extreme rays in dimension 4, 
with $t_i=i$ for $i\in [5]$ in the definition of Example~\ref{Example2}, 
i.e. if $\sK$ is the set of tropical linear combinations of the rows of the matrix
\[
G=
\left(\begin{array}{cccc}
0&1&2& 3   \\
0&2&4& 6   \\
0&3&6& 9   \\
0&4&8& 12  \\
0&5&10&15  
\end{array}\right) \enspace ,
\]
applying the previous algorithm we get the following minimal 
system of inequalities defining $\sK$:

\[
\begin{array}{rlccrl}
-1x_2  &\leq x_1   \oplus -3x_3 &&&
-2x_2  &\leq x_1   \oplus -5x_3 \\
-3x_2  &\leq x_1   \oplus -7x_3&&&
-15x_4 &\leq x_1 \\
-2x_3  &\leq  x_2 \oplus -5x_4 &&&
-3x_3  &\leq x_2 \oplus -7x_4 \\
-4x_3  &\leq x_2 \oplus -9x_4&&&
-5x_3  &\leq x_2 \\
-5x_4  &\leq x_3 &&&
 x_1   &\leq -1x_2 \\
 x_1   &\leq -2x_3&&&
 x_2   &\leq -2x_4\\
-4x_2  &\leq  x_1\oplus -9x_3 &&&
-1x_3  &\leq x_2 \oplus -3x_4 
\end{array}
\]
It can be checked that by replacing the four inequalities on the last two rows
by
\[
\begin{array}{rlccrl}
 x_2   &\leq -1x_3 &&&
 x_3   &\leq -1x_4\\
-4x_2  &\leq  x_1 \oplus -14x_4 &&&
-2x_3  &\leq  x_1 \oplus -4x_4
\end{array}
\]
we still get a minimal defining system. 

The following vectors are certificates that the first system of $14$ inequalities above is minimal: 
each vector satisfies all the inequalities but the one on the same row and column. 
\[
\begin{array}{cc}
(0,2,2,4) & (0,3,5,7) \\
(0,4,7,10) & (\mpzero ,0,5,10) \\
(0,3,6,8)  & (0,4,8,11) \\
(0,6,11,15)  & (0,1,10,15) \\
(0,1,2,15)  & (0,0,2,4) \\
(2,3,0,5)   & (0,1,2,\mpzero ) \\
(0,7,11,15)  & (0,2,4,4) 
\end{array}
\]
For the second system of inequalities, the certificates are: 
\[
\begin{array}{cc}
(0,2,3,5) & (0,3,5,7) \\
(0,4,7,10) & (\mpzero ,0,5,10) \\
(0,3,6,8)  & (0,4,8,11) \\
(0,4,9,13)  & (0,1,10,15) \\
(0,1,2,15)  & (0,\mpzero ,\mpzero ,\mpzero ) \\
(0,1,\mpzero ,\mpzero )  & (0,1,2,\mpzero ) \\
(0,7,11,15)  & (0,2,4,5) 
\end{array}
\]

Let us finally give more details on how the previous inequalities were obtained. We know
from Proposition~\ref{PropExtCyclic} that the polar cone of $\mathcal{K}$
has precisely $36$ extreme rays. Excluding the $8$ inequalities
$x_i\geq x_i$ and $x_i\geq \mpzero$, for $i\in[4]$, we have $28$
non-trivial inequalities. 
The latter can be obtained by enumerating the 
corresponding tropically allowed lattice paths, 
following the proof of Proposition~\ref{PropExtCyclic},
or using directly the tropical polyhedral library TPLib~\cite{TPLib}.
These $28$ inequalities consist of the $18$ inequalities listed in the two
groups above, together with the following $10$ inequalities:
\[\begin{array}{rlccrl}
 3x_1 &\leq x_4&&&
 4x_2 &\leq 5x_1 \oplus x_4\\
 3x_3 &\leq 7x_1 \oplus x_4
&&&
 4x_3 &\leq 10x_1 \oplus x_4\\
 5x_3 &\leq 13x_1 \oplus x_4&&&
 6x_2 &\leq 8x_1 \oplus x_4\\
 8x_2 &\leq 11x_1 \oplus x_4&&&
 -10x_3 &\leq x_1\\
 -10x_4 &\leq x_2&&&
 -5x_2 &\leq x_1
\end{array}
\]
Then, we eliminated successively redundant inequalities, by using
the previously mentioned greedy method, in which we used
the value iteration algorithm of~\cite{zwick}, or rather its variant in~\cite{AGG}, to compute $\rho(g_1)$ at each step.

\newcommand{\etalchar}[1]{$^{#1}$}
\def\cprime{$'$} \def\cprime{$'$}

\end{document}